\documentclass[reqno,twoside,11pt]{amsart}

\usepackage{amsmath,amsfonts,calrsfs,fullpage,amssymb,color}

\newtheorem{Theorem}{Theorem}[section]
\newtheorem{Definition}[Theorem]{Definition}
\newtheorem{Proposition}[Theorem]{Proposition}
\newtheorem{Lemma}[Theorem]{Lemma}
\newtheorem{Corollary}[Theorem]{Corollary}
\newtheorem{Remark}[Theorem]{Remark}

\newtheorem{Hypothesis}[Theorem]{Hypothesis}
\makeatletter
\@addtoreset{equation}{section}

\makeatother

\def\R{\mathbb R}
\def\N{\mathbb N}

\def\C{\mathbb C}
\def\E{\mathbb E}

\def\C{\mathcal C}
\def\eps{\varepsilon}
\def\ds{\displaystyle}

\newcommand{\dist}{\operatorname{dist}}
\newcommand{\Tr}{\operatorname{Tr}}

\newcommand{\one}{1\!\!\!\;\mathrm{l}}

\title[Maximal Sobolev regularity in infinite dimension]{\bf Maximal Sobolev regularity in Neumann problems for gradient systems in infinite dimensional domains}

\author[G. Da Prato]{Giuseppe Da Prato}
\address{Scuola Normale Superiore\\
Piazza dei Cavalieri, 7\\ 
56126 Pisa, Italy}
\email{g.daprato@sns.it}

\author[A. Lunardi]{Alessandra Lunardi}
\address{
Dipartimento di Matematica\\
Universit\`a di Parma\\
Parco Area delle Scienze, 53/A\\
43124 Parma, Italy}
\email{alessandra.lunardi@unipr.it}



\begin{document}

 \begin{abstract}  
We consider an elliptic Kolmogorov equation $\lambda u - Ku = f$ in a convex subset $\mathcal C$ of a separable Hilbert space $X$. The Kolmogorov operator $K$ is a realization of   $u\mapsto \frac12\mbox{Tr}\;[D^2u(x)]+ \langle Ax - DU(x),Du(x)\rangle$,  $A $ is a self--adjoint operator in $X$ and $U:X\mapsto \R \cup \{+\infty\}$ is  a convex function. We prove that for $\lambda >0$ and  $f\in L^2(\mathcal C,\nu)$
the weak solution $u$ belongs to the Sobolev space $W^{2,2}(\mathcal C,\nu)$, where $\nu$ is the log-concave measure associated to  the system. Moreover we prove maximal estimates on the gradient of $u$, that allow to show that $u$ satisfies the Neumann boundary condition in the sense of traces at the boundary of $\mathcal C$.  The general results are applied to Kolmogorov equations of reaction--diffusion and Cahn--Hilliard stochastic PDEÕs in convex sets of suitable Hilbert spaces. 
 \end{abstract}

\maketitle

 
\section{Introduction} 
 Let $X$ be an infinite dimensional separable Hilbert space, with norm $\|\cdot\|$ and scalar product $ \langle\cdot,\cdot\rangle$. We study  the Neumann problem for the differential equation
  \begin{equation}
  \label{Neumann}
  \lambda u-\frac12\mbox{Tr}\;[D^2u]-\langle Ax-DU(x),Du\rangle=f,\quad x\in \C .
  \end{equation} 
Here,  $A:D(A)\subset X\to X$ is a linear self--adjoint operator, strictly negative and such that $A^{-1}$ is of trace class, $U: X\to \R \cup \{+\infty\}$ is a convex function, and $\C$  is a convex closed subset of $X$.  Moreover, $\lambda>0$ and $f:X\to\R$ are given data. $Du$ and $D^2u$ represent the gradient and the Hessian of $u$ and $\mbox{Tr}\;[D^2u]$ the trace of $D^2u$.

The Neumann problem for equation \eqref{Neumann} can be considered as the Kolmogorov equation corresponding to  the   stochastic  
variational inequality with reflection 
\begin{equation}
\label{e11}
\left\{\begin{array}{l}
dX(t,x)-A X(t,x)\,dt + N_{\mathcal C}(X(t))dt \ni dW(t),\\\\
X(0)=x, 
\end{array}\right. 
\end{equation}
where    $N_{\mathcal C}$ is the normal cone to  ${\C}$ and $W(t)$ is a $X$-valued cylindrical Wiener process. This is because,  at least formally, we have
\begin{equation}
\label{e10}
u(x)=\int_0^\infty e^{-\lambda t}\E[f(X(t,x))]dt,\quad x\in X.
\end{equation}

In the case that $X$ is finite dimensional a quite general theory of stochastic  
variational inequalities  with maximal monotone coefficients was developed by  E.  C\'epa \cite{C98}, who proved  existence  and uniqueness of a solution  $X(\cdot,x)$ to \eqref{e11} and established  its connection with the celebrated Skorokhod  problem.
The fact that the function $u $ given by formula \eqref{e10} fulfills the Neumann boundary condition on $\partial \,\mathcal C$ was proved in \cite{BDP08}.

In infinite dimensions the situation is more delicate. 
The first important result is in the seminal paper by E. Nualart and E. Pardoux \cite{NP}, who solved a  reaction-diffusion problem with reflection in $X=L^2(0,1),$
\begin{equation}
\label{e2}
\left\{\begin{array}{l}
dX(t,x)-\Delta X(t,x)\,dt + f(X(t,x))\,dt + N_{\mathcal K}(X(t))\,dt \ni dW(t),\\\\
X(0)=x,
\end{array}\right. 
\end{equation}
where  $f$ is decreasing and   $N_{\mathcal K}$ is the normal cone to the set  $\mathcal K$   of nonnegative functions.
Then,  L. Zambotti  \cite{Z}  exhibited an explicit (unique) invariant measure $\mu$, and  proved the existence of a unique weak solution in $L^2(X,\mu)$ to \eqref{Neumann}, as well as basic integration by parts formulae   on 
$ \mathcal K$ (note that the interior part of $ \mathcal K$ is empty). Related results, applied to some interface problem, were provided by T. Funaki and S. Olla \cite{FO}. 
A part of these results  have been extended by A. Debussche and L. Zambotti \cite{DZ} to the  reflection problem  for  a  Cahn--Hilliard equation again on a suitable convex set of nonnegative functions.

Later on the study of \eqref{e2} and \eqref{Neumann} was pursued, using Lagrangian flows  by L. Ambrosio, G. Savar\'e and L. Zambotti in \cite{ASZ}, and using Dirichlet forms  by M. R\"ockner,  R.-C. Zhu, X.-C. Zhu in  \cite{Z}.
In  these papers,  among other results,  existence  and uniqueness of a weak solution of \eqref{Neumann} where established, but further regularity and  existence of a vanishing normal derivative on the boundary remained open problems.

For  smooth convex sets and for the Ornstein--Uhlenbeck equation with $U\equiv 0$, problem  \eqref{Neumann} was studied by V. Barbu, G. Da Prato and L. Tubaro in \cite{BDPT1,BDPT2}, extending to the infinite dimensional setting a  penalization argument already used in the finite dimensional case (e.g., \cite{DPL}) and referring to the Airault--Malliavin theory of infinite dimensional surface measures  \cite{AM}.  They   showed that the weak solution of  \eqref{Neumann}   is in a Sobolev space $W^{2,2}({\mathcal C}, \mu)$, where  $\mu$ is the Gaussian measure with mean $0$ and covariance $Q=-\frac12\,A^{-1}$, which is symmetrizing (and hence, invariant) for the Ornstein--Uhlenbeck operator in the equation. They also addressed the Neumann boundary condition in the sense of traces at the boundary of Sobolev functions.
However their proof is not convincing,  and a first goal of our paper is to provide a complete proof of the Neumann condition. Our proof too uses penalization  and provides maximal regularity estimates for equation \eqref{Neumann}  as in \cite{BDPT1,BDPT2}, but  the proof of existence and vanishing at the boundary  of the normal derivative of the solution $u$ is completely different. Besides the  regularity of the second derivative of $u$, we  use  in essential way another maximal regularity result, namely that $\|(-A)^{1/2}Du\|$ is in $L^2$ (a fact also proved but not exploited for the existence of the normal derivative in \cite{BDPT1,BDPT2}), as well as the recent study of traces of Sobolev functions
on hypersurfaces by P. Celada and A. Lunardi \cite{Tracce}. 
 
The second goal of our paper is to study problem \eqref{Neumann} for a broad class of convex potentials $U \not\equiv 0$.
The extension of the regularity theory to this case is not straightforward.  The relevant invariant measure is the log--concave measure $\nu(dx):= e^{-2U(x)}\mu(dx)$, where  $\mu$ is still the Gaussian measure of mean $0$ and covariance  $Q=-\frac12\,A^{-1}$.

Let us give more details about the contents of the paper and the encountered difficulties. 

In Section 2 the Sobolev spaces $W^{1,p}(\C,\nu)$, $W^{2,p}(\C,\nu)$ are  defined, in such a way that the self--adjoint operator $K$ canonically associated to the quadratic form $\mathcal E(u,v) = \frac12\int_{\C}\langle Du, Dv\rangle d\mu$  in $ W^{1,2}(\C,\nu)$, is a realization of   the Kolmogorov operator $u\mapsto \frac12\mbox{Tr}\;[D^2u] -\langle Ax + DU(x),Du\rangle$ in $L^{2}(\C,\nu)$. A function $u\in W^{1,2}(\C,\nu)$  is called weak solution to the Neumann problem for \eqref{Neumann}
if
$$\int_{\C}\lambda\,u\,\varphi \, d\nu= \frac12\int_{\C}\langle Du, D\varphi\rangle d\nu + \int_{\C} \varphi\,f\,  d\nu , \quad \forall \varphi\in W^{1,2}(\C,\nu).$$
It is not difficult to see that  for every $\lambda >0$ and $f\in L^{2}(\C,\nu)$ the Neumann problem has a unique weak solution $u$, which is just $R(\lambda, K)f$. 

Further properties of the weak solution are studied in Section \ref{sect_Neumann}. In \S \ref{Regularity}
we prove that   $u$ belongs to $W^{2,2}(\C,\nu)$, and that $\int_\C\|(-A)^{1/2}Du\|^2d\nu<\infty$.  In the case $\C =X$ this was already shown in \cite{DPL0}, and in fact for the proof we use some results from \cite{DPL0}. 
Indeed, as in the finite dimensional case \cite{DPL},  we consider a family of penalized problems in the whole space $X$, with  $\alpha>0$, 
   \begin{equation}
  \label{e8a}
  \lambda u_\alpha-\frac12\mbox{Tr}\;[D^2u_\alpha]-\langle Ax-DU_\alpha(x),Du_\alpha\rangle+\frac1\alpha\;\langle x-\Pi_{ \C}(x),Du_\alpha\rangle=f,
  \end{equation} 
where   $U_\alpha$ are suitable approximations of $U$ and $\Pi_{ \C}(x)$ is the projection of $x$ on $ \C$. Setting $\nu_{\alpha}(dx):= e^{-2U_{\alpha}(x)}\mu(dx)$ and using the estimates of \cite{DPL0} for equations in the whole space,   we show that the restrictions of $u_\alpha $ to $\C$ are bounded in $W^{2,2}(\C, \nu_{\alpha})$ by a constant independent of $\alpha$, and that $\int_\C\|(-A)^{1/2}Du_{\alpha}\|^2d\nu_{\alpha} $ is bounded by a constant independent of $\alpha$. These estimates are the key ingredients to
 obtain the desired result, letting $\alpha\to 0$. 
 
In the case that $U$ belongs to a suitable Sobolev space, we can take as $U_\alpha$ the Moreau--Yosida approximations of $U$ (note that $( x-\Pi_{\overline{\C}}(x))/\alpha$ is the gradient of the Moreau--Yosida approximations of the characteristic function of  $ \C$) as in \cite{DPL} and in \cite{BDPT1,BDPT2}. 
However there are interesting examples, such as the Kolmogorov equations of stochastic Cahn--Hilliard equations considered in \S \ref{Cahn--Hilliard}, for which $U$ has not sufficient Sobolev regularity, and we have to make other approximations. 

The  Neumann boundary condition is discussed in \S \ref{Neu_bordo}. We assume that $\mathcal C = \{x:\;G(x) \leq 0\}$, where $G$ is a fixed version of a nondegenerate Sobolev function, belonging to suitable $W^{2,p}$ spaces. The theory of traces of Sobolev functions with respect to Gaussian measures at level sets of  $G$ was recently addressed in \cite{Tracce}; here we extend parts of it to the case of the weighted Gaussian measure $\nu$. The traces belong to weighted Lebesgue spaces with respect to the Hausdorff--Gauss surface measure $\rho$ of \cite{FP}, naturally associated to the Gaussian measure $\mu$.  
We are interested in the level set $G=0$, which is the boundary of $\mathcal C$ if $G$ is continuous.  
The Neumann boundary condition is meant as 
\begin{equation}
\label{Neumanneffettiva}
\langle Du, DG \rangle =0\quad {\rm at}\;G^{-1}(0) ,
\end{equation}
in the sense of traces of Sobolev functions. Of course we need that $\langle Du, DG \rangle $ is a Sobolev function, and it is here that the estimate 
 $\int_\C\|(-A)^{1/2}Du\|^2d\nu<\infty$ is used. 
   
The last section of the paper contains examples of admissible sets $\mathcal C$, and two applications to Kolmogorov equations of stochastic PDE's. 
The first one is a reaction--diffusion equation in $X=L^2(0, 1)$, with polynomially growing nonlinearity $x\mapsto F\circ x$. It corresponds to the Nualart--Pardoux reflection problem, with a more regular closed convex set replacing the set of nonnegative functions considered in \cite{NP}. 
The second one is the Cahn--Hilliard  equation considered in \cite{DZ}. Here the nonlinearity is   $x\mapsto \partial ^2/\partial \xi^2(F\circ x)$. For such a nonlinearity be of gradient type, we choose a  Sobolev space of negative order as a reference space $X$. Again, the set of nonnegative functions is replaced by a more regular convex set. 

As expected, the infinite dimensional case exhibits extra difficulties and different features with respect to the finite dimensional case treated in \cite{DPL}.  For instance the condition $\int_{\C}\|(-A)^{1/2}Du\|^2d\nu<\infty$  is   satisfied  by any $u\in  W^{1,2}(\C,\nu)$ in finite dimensions. Instead, in infinite dimensions this extra estimate is significant, and it is crucially used to prove that $u$ satisfies the Neumann boundary condition. Moreover Sobolev functions have continuous versions  in finite dimensions, so that there are not  difficulties due to the possible discontinuities of $G$; in particular $G^{-1}(0)$ is just the boundary of $\C$. In infinite dimensions we consider a fixed quasicontinuous (in the sense of Gaussian capacities, see sect. 
\ref{Neu_bordo}, \cite[Sect. 5.9]{Boga}) version of $G$ and everything goes through, paying  the price of more technicalities to deal with.

It would be interesting to generalize our results to less regular convex sets, as the ones considered in \cite{NP} and \cite{DZ}. For the moment, the main obstacles are the regularity requirements of the trace theory  from \cite{Tracce}.

\section{Notation and preliminaries}
\label{Notation}

Let $X$ be a separable Hilbert space endowed with a Gaussian measure $\mu := {\mathcal N}_{\;0, Q}$ of mean $0$ and covariance operator $Q$, where  $Q\in {\mathcal L}(X)$ is   self-adjoint, strictly positive, and  with finite trace. We choose once and for all an orthonormal basis $\{ e_k:\;k\in \N\}$ of $X$ such that  $Qe_k = \lambda_k e_k$ for $k\in \N$.  We denote by  $P_n$   the orthogonal projection on the linear span of $e_1, \ldots, e_n$. 

For each $k\in \N \cup \{+\infty\}$ we denote by   ${\mathcal F}{\mathcal C}^k_b(X)$  the set of the cylindrical functions  $\varphi(x) = \phi(x_1, \ldots, x_n)$ for some $n\in \N$, with $\phi\in C^k_b(\R^n)$.

 
\subsection{Sobolev spaces}
\label{Sobolev spaces}


\subsubsection{Sobolev spaces with respect to $\mu$}

If a function $\varphi:X\mapsto \R$ is differentiable at $x\in X$, we denote by $D\varphi (x)$ its gradient at $x$. 

For $  \theta \in \R$ and $p\geq1$ the Sobolev spaces $W^{1,p}_{\theta}(X,\mu)$  are  the completions of ${\mathcal F}{\mathcal C}^1_b(X)$   in the Sobolev norms
$$\|\varphi \|_{W^{1,p}_{\theta}(X, \mu)}^p : =  \int_X (|\varphi |^p +  \|Q^{\theta}D\varphi\|^p)d\mu = 
\int_X \bigg(  |\varphi |^p +  \bigg(\sum_{k=1}^\infty (\lambda^{ \theta}_k  D_k\varphi )^2\bigg)^{p/2} \bigg)d\mu .$$
For $\theta=1/2$ they coincide with the usual Sobolev spaces of the Malliavin Calculus, see e.g. \cite[Ch. 5]{Boga}; for $\theta =0  $ and $p=2$ they are the spaces considered in \cite{DPZ3}.  Such  completions are identified with subspaces of $L^{p}(X,\mu)$ since the integration by parts formula 
\begin{equation}
\label{partimuX}
 \int_X D_k\varphi\,\psi\,d\mu= -\int_X D_k\psi\,\varphi\,d\mu  +\frac{1}{\lambda_k}\int_X x_k\varphi\,\psi\,d\mu, \quad \varphi, \;\psi \in {\mathcal F}{\mathcal C}^1_b(X),
 \end{equation}
allows easily to show that the operators $Q^{\theta}D: {\mathcal F}{\mathcal C}^1_b(X)\mapsto 
L^p(X, \mu; X)$ are closable in $L^{p}(X,\mu)$, and the domains of their closures (still denoted by $Q^{\theta}D$) coincide with $W^{1,p}_{\theta}(X,\mu)$.

\subsubsection{Sobolev spaces with respect to $\nu$}

We  shall assume that  $U: X\mapsto \R \cup \{+\infty \}$  is a convex function that can be approximated by a family of nice functions $U_{\alpha}$. Precisely, 

\begin{Hypothesis}
\label{Hyp}
$U: X\mapsto \R \cup \{+\infty \}$  is  convex. There are functions $U_{\alpha}:X \to \R$, $\alpha >0$, with the following properties. 
\begin{itemize}
\item[(i)] Each $U_{\alpha}$ is convex, differentiable at every $x\in X$, and $DU_{\alpha}$ is Lipschitz continuous; 
\item[(ii)] $\exists C\in \R :$ $C\leq U_{\alpha}(x) \leq U(x)$ for every $\alpha >0$,  a.e. $x\in X$; 
\item[(iii)] there exists $p_0>2$ such that $\lim_{\alpha\to 0}U_{\alpha} = U$ in $W^{1,p_0}_{1/2}(X, \mu)$.
\end{itemize}
\end{Hypothesis}

Since each  $U_{\alpha}$ is continuously differentiable and  has Lipschitz continuous gradient,  then $U_{\alpha} \in W^{1,q}(X,\mu)$ for every $q$. This can be easily proved arguing as in the case $q=2$ of \cite[Prop. 10.11]{Beppe}. Moreover, taking into account that both $U$ and $U_{\alpha}$ have a.e. values in $[C, +\infty)$, where the function $\xi\mapsto e^{-2\xi}$ is bounded and Lipschitz continuous, we obtain easily that $e^{-2U_{\alpha}} $ converges to $e^{-2U}$ in $W^{1,p_0}_{1/2}(X,\mu)$ as $\alpha \to 0$.

Note that the heaviest requirement in Hypothesis \ref{Hyp} is that $DU_{\alpha}$ is Lipschitz continuous. The other ones are satisfied by any convex $U\in W^{1,p_0}_{1/2}(X, \mu)$, such that 
$U(x)\geq C$ for a.e. $x\in X$. 

We describe here a (large enough) class of functions $U$ that satisfy Hypothesis \ref{Hyp}. Let $U:X \to \R\cup \{+\infty\} $ be convex, bounded from below, and lower semicontinuous. 
For $\alpha >0$ we denote by  $U_{\alpha}$  the Moreau--Yosida approximations of $U$,  defined by  
\begin{equation}
\label{Moreau}
U_\alpha(x): =\inf\left\{U(y)+\frac{|x-y|^2}{2\alpha},\;y\in X\right\} .
\end{equation}
Then, (i) is satisfied, and   $U_{\alpha}(x)$ converges monotonically  to $U(x)$ for each $x$ as $\alpha \to 0$, so that (ii) is satisfied too. Moreover, denoting by $D_0U(x)$ the element with minimal norm in the subdifferential  of $U(x)$,  at any $x$ such that the subdifferential of $U(x)$ is not empty,  $\|DU_{\alpha}(x)\|$  converges monotonically  to $\|D_0U(x)\|$.  At such points we have
\begin{equation}
\label{M-Y}
 \|DU_{\alpha} - D_0U\|^2 \leq \|D_0U\|^2 -  \|DU_{\alpha} \|^2 .
\end{equation}
See e.g. \cite[Ch. 2]{Brezis}.

\begin{Lemma}
\label{regU-Moreau}
Let $U:X \to \R\cup \{+\infty\} $ be convex, bounded from below, and lower semicontinuous. Assume in addition that $x\mapsto \|D_0U(x)\| \in L^{p_1} (X, \mu)$ for some $p_1>1$. Then $U\in W^{1,p_0}_{0}(X,\mu)$ for each $p_0<p_1$,  $DU=D_0U$ a.e., and $\lim_{\alpha \to 0}U_{\alpha} =U$ in $W^{1,p_0}_{0}(X,\mu)$; consequently $U\in W^{1,p_0}_{\theta}(X,\mu)$ and $\lim_{\alpha \to 0}U_{\alpha} =U$ in $W^{1,p_0}_{\theta}(X,\mu)$ for every $\theta \geq 0$. 
\end{Lemma}
\begin{proof}
Let us prove that $U\in L^q(X, \mu)$ for $ p_0<p_1$. For $\mu$-a.e. $x\in X$ and for each $y\in X$ we have $ U(y)-U(x) \geq \langle D_0U(x), y-x\rangle $  by the convexity assumption, so that 
$$C \leq U(x) \leq U(y) - \langle D_0U(x), y-x\rangle  \leq U(y) + \|D_0U(x)\| (\|y\| + \|x\|).$$
Fix any $y$ such that $U(y)<\infty$. Since $x\mapsto \|x\| \in L^r(X, \mu)$ for every $r$, by the H\"older inequality 
$U\in L^{p_0}(X, \mu)$ for $ p_0<p_1$. 

Let us  prove that $U\in W^{1,p_0}_{0}(X,\mu)$. Recall that  $U_{\alpha} \in W^{1,p}_{0}(X,\mu)$ for every $\alpha >0$, $p>1$.  By dominated convergence $U_{\alpha}\to U$ in $L^{p_0}(X, \mu)$,  and $\| DU_{\alpha}- D_0U\| \to 0$ in $L^{p_0}(X, \mu)$, since $  C \leq U_{\alpha} \leq U$ and $ \|DU_{\alpha} - D_0U\|  \leq \|D_0U\|$. 
Then, $U\in W^{1,p_0}_{0}(X,\mu)$ and $DU= D_0U$ a.e. 
\end{proof}

If $U$ satisfies the assumptions of Lemma \ref{regU-Moreau} with $p_1>2$ then its Moreau--Yosida approximations satisfy Hypothesis \ref{Hyp}. 
However, there are important examples such that  $U\notin  W^{1,p_0} (X, \mu)$  for any  $p_0>1$. 
We shall see  one of such examples in \S \ref{Cahn--Hilliard}, where the Moreau--Yosida approximations will be replaced by other {\em ad hoc} approximations. 

We denote by     $\nu$ the log-concave measure
\begin{equation}
\label{e3}
\nu(dx)= e^{-2U(x)} 
\mu(dx),
\end{equation}
Since $e^{-2U } $ is bounded, $\nu(X)< +\infty$.

By Lemma \ref{regU-Moreau}, we may apply the integration by parts formula \eqref{partimuX} with   $\psi e^{-2U}$ replacing $\psi$, that belongs to $W^{1,p_0}_{1/2}(X,\mu)$ for $\psi \in {\mathcal F}{\mathcal C}^1_b(X)$. We get, for $\varphi$, $\psi\in   {\mathcal F}{\mathcal C}^1_b(X)$ and $h\in \N$, 
\begin{equation}
 \label{partinuX}
\int_X D_h\varphi\,\psi\,d\nu+\int_X D_h\psi\,\varphi\,d\nu= -2 \int_X
D_hU\,\varphi\,\psi  \,d\nu+\frac1{\lambda_h}\int_X x_h \varphi\,\psi\,d\nu   .
 \end{equation} 

Once again, the Sobolev spaces  associated to the measure $\nu$ are introduced in a standard way with the help of the  integration by parts formula \eqref{partinuX}. We recall that  $ {\mathcal L}_2(X)$ is the space of the Hilbert--Schmidt operators, that is the bounded linear operators  $L:X\mapsto X$ such that $\|L\|_{{\mathcal L}_2(X)}^2 :=\sum_{h,k}\langle Le_h, e_k\rangle^2 <\infty$. 

 For $p>1$ we set as usual $p' : = p/(p-1)$. In the paper \cite{DPL0} we proved that for all   $q\geq p_0'$ the operators
$$D: {\mathcal F}{\mathcal C}^1_b(X)\mapsto 
L^q(X, \nu; X), \quad Q^{\pm 1/2}D: {\mathcal F}{\mathcal C}^1_b(X)\mapsto 
L^q(X, \nu; X)$$
$$( D,  D^2 ): {\mathcal F}{\mathcal C}^2_b(X) \mapsto L^{q}(X,\nu; X)\times L^{q}(X,\nu; {\mathcal L}_2(X))$$
are closable. Their closures were still denoted by 
 $ D$, $Q^{1/2} D$, $Q^{-1/2} D$, and by 
$(D,  D^2 )$, respectively. The Sobolev spaces 
 $W^{1,q} (X, \nu)$, $W^{1,q}_{1/2}(X, \nu)$, $W^{1,q}_{-1/2}(X, \nu)$, were defined as  the domains of $ D$, $Q^{1/2} D$, $Q^{-1/2} D$ in $L^q(X, \nu)$, respectively. The space  $W^{2,q}(X, \nu)$ was defined as the domain of $( D,  D^2 )$ in $L^q(X, \nu)$.

The Sobolev spaces on general subsets $  \C \subset X$ may be defined in several ways. The most convenient for our purposes relies on the following lemma, which allows to extend the above definitions to the case of a subset $\C\subset X$, $\C\neq X$. 

\begin{Lemma}
\label{Le:chiusura}
Let $G$ be fixed version of an element of  $\cap_{p>1}W^{1,p}_{1/2}(X,  \mu)$, and set $\C = G^{-1}((-\infty, 0])$. Then for every  $\theta\in \R$ and  $q\geq p_0'$  the operators  $ D: {\mathcal F}{\mathcal C}^1_b(X)\mapsto L^q(\C, \nu; X)$, $u\mapsto  Du_{|\C}$, $Q^{\theta}D: {\mathcal F}{\mathcal C}^1_b(X)\mapsto  L^q(\C, \nu; X)$, 
$u\mapsto Q^{\theta}Du$ are closable. 
Their closures are still denoted by $ D$, $Q^{ \theta}D$, respectively. 

Similarly, the operator $( D,  D^2 ): {\mathcal F}{\mathcal C}^2_b(X)$ $ \mapsto L^q(\C, \nu; X)$ $\times $ $L^{q}(\C,\nu; {\mathcal L}_2(X))$, $u\mapsto ( Du_{|\C}, D^2 u_{|\C})$ is closable. The closure is still denoted by $( D, D^2 )$. 
\end{Lemma}
\begin{proof}
Let  $u_k \in {\mathcal F}{\mathcal C}^1_b(X)$ be such that  $u_k\to 0$ and  $ Du_k\to \Phi$ (respectively, $Q^{  1/2}Du_k\to \Phi$, $Q^{ -1/2}Du_k\to \Phi$) in $L^q(\C, \nu; X)$ as $k\to \infty$. We claim that  $\int_{\C} \langle \Phi, e_i\rangle \, \psi \,d\nu =0$ for each  $\psi \in  L^{q'}(\C, \nu)$ and $i\in \N$. Since the restrictions to $\C$ of the elements of  $C^1_b(\C)$ are dense in $ L^{q'}(\C, \nu)$ (as a consequence of the density of  $C^1_b(X)$ in $L^{q'}(X, \nu)$) it is sufficient to prove that 
\begin{equation}
\label{nullo}
\int_{\C}  \langle \Phi, e_i\rangle \, \psi \,d\nu  =0, \quad \psi \in  C^1_b(X).
\end{equation}
To this aim, for every $\psi \in  C^1_b(X)$ we approach its restriction to $\C$    by restrictions to  $\C$ of elements of  $W^{1,q'}_{1/2} (X, \mu)$ that vanish outside $  \C$. This is to reduce integrals over $\C$ to integrals over $X$, avoiding surface integrals in the next integration by parts. 
We fix a function  $\theta \in C^{\infty}_c(\R)$ such that  $\theta(r) =0$ for  $r\geq -1$, $\theta(r) =1$ for $r\leq -2$, and we set   
$$\theta_n(r):= \theta (nr), \quad \psi_n(x) := \psi (x)\theta_n(G(x)), \qquad n\in \N, \;x\in X.$$
By dominated convergence the sequence $(\psi_{n|\C})$ goes to   $\psi _{|\C}$ in $L^{q'}(\C, \nu)$ as $n\to \infty$. Then, 
$$\int_{\C}  \langle \Phi, e_i\rangle \, \psi \,d\nu = \lim_{n\to \infty} \int_{\C} \langle \Phi, e_i\rangle \, \psi_n \,d\nu .$$
Moreover each  $\psi_n$ vanishes in $G^{-1}([-1/n, +\infty))$, and  $ Q^{1/2}D\psi_n =   Q^{1/2}D\psi\,\theta_n\circ G + \psi \theta_n'\circ G\,  Q^{1/2}DG$, so that  $ \psi_n$ belongs to  $  W^{1,r}_{1/2} (X, \mu)$ for every $r>1$. 
It follows that for each   $k\in \N$,   $u_k\psi_n$  belongs to  $ W^{1,r} _{1/2}(X, \mu)$  for every $r>1$. 

Taking into account that $e^{-2U}\in W^{1,p_0}_0(X, \mu) \subset W^{1,p_0}_{1/2}(X, \mu)$, 
the integration by parts formula 
\eqref{partimuX} gives
$$\int_{\C} D_i(u_k\psi_n) d\nu = \int_{X} D_i(u_k\psi_n)  \;e^{-2U}d\mu =  2\int_{\C}  u_k\psi_n\, D_iU \,d\nu + 
\frac{1}{\lambda_i}  \int_{\C} x_i u_k\psi_n\,d\nu ,$$
so that 
\begin{equation}
\label{serve}
\int_{\C} \psi_n \,D_iu_k\,d\nu = -\int_{\C} u_kD_i\psi_n\,d\nu + 2\int_{\C}  u_k\psi_n\, D_iU \,d\nu + \frac{1}{\lambda_i}  \int_{\C} x_i u_k\psi_n\, d\nu
\end{equation}
and letting $k\to \infty$ the right hand side converges to  $0$, and the left hand side converges to   $\lambda_i^{-\theta}\int_{\C} \psi_n  \langle \Phi, e_i\rangle  d\mu$ (it is here that we need $q\geq p_0'$). Then, $\int_{\C} \psi_n\langle \Phi, e_i\rangle d\mu =0$ for each  $n$ and  \eqref{nullo} holds. 

The proof of the second part of the statement is similar. In this case we have a sequence 
$u_k \in {\mathcal F}{\mathcal C}^2_b(X)$   such that  $u_k\to 0$ and  $ Du_k\to \Phi$ in $L^q(\C, \nu; X)$,  $ D^2u_k \to {\mathcal Q}$ in $L^q(\C, \nu; {\mathcal L}_2(X))$
as $k\to \infty$. By the first part of the proof, $\Phi =0$. Moreover, formula 
\eqref{serve} applied to $D_ju_k $ instead of $u_k $ gives
$$ \int_{\C} \psi_n\,D_{ij}u_k\, d\nu   = 
  -\int_{\C} D_ju_k\,D_i\psi_n\,d\nu + 2\int_{\C}  D_ju_k\,\psi_n\, D_iU \,d\nu + \frac{1}{\lambda_i}  \int_{\C} x_i D_ju_k\,\psi_n\, d\nu$$
where the left hand side converges to   $\int_{\C} \psi_n \langle {\mathcal Q}e_i, e_j\rangle  d\mu$ and the right hand side converges to  $0$. Then, $\int_{\C} \psi_n  \langle {\mathcal Q}e_i, e_j\rangle  d\mu =0$ for each $n$, so that for every $\psi\in C^1_b(X)$ we have
$$\int_{\C} \psi  \langle {\mathcal Q}e_i, e_j\rangle  d\mu = \lim_{n\to \infty}  \int_{\C} \psi_n  \langle {\mathcal Q}e_i, e_j\rangle  d\mu =0, $$
which implies that ${\mathcal Q}=0$. 
 \end{proof}

The Sobolev spaces   $W^{1,p}_{\theta}(\C, \nu)$ and $W^{2,p} (\C, \nu)$ for $p\geq  p_0'$ are defined as the domains of the closures of the above operators. 
For $p=2$ they are Hilbert spaces with the scalar products
$$\langle u, v\rangle_{W^{1,2}_{\theta}(\C, \nu)} :=  \int_{\C} (u\,v + \langle Q^{\theta}D u, Q^{\theta}D v\rangle_X) d\nu ,$$
$$\langle u, v\rangle_{W^{2,2} (\C, \nu)} :=  \langle u, v\rangle_{W^{1,2}_{\theta}(\C, \nu)} + \int_{\C}  \langle  D^2u , D^2v\rangle_{{\mathcal L}_2(X)}\,d\nu.$$

Of course,  $W^{1,p}_{-1/2}(\C, \nu)\subset W^{1,p}_0 (\C, \nu) \subset W^{1,p}_{1/2}(\C, \nu)$ for every $p\geq p_0'$.

Note that if $G_1=G_2$ a.e., the symmetric difference of the sets $\C_1= G_1^{-1}((-\infty , 0])$ and $\C_2= G_1^{-1}((-\infty , 0])$ is negligible, and the above defined Sobolev spaces on $\C_1$ and $\C_2$ coincide. 

It is not our aim here to develop a complete theory of Sobolev spaces. We just mention some  properties that  will be used in the sequel. 

\begin{Proposition}
\label{varie}
Let $p\geq p_0'$, and let $\theta\in \R$. Then
\begin{itemize} 
\item[(i)] $W^{1,p}_{\theta}(\C, \nu)$ is reflexive; 
\item[(ii)] if a bounded sequence of elements of $W^{1,p}_{\theta}(\C, \nu)$ converges a.e. in $\C$ to a function $f$, then $f\in W^{1,p}_{\theta}(\C, \nu)$; 
\item[(iii)] if $f\in W^{1,2}_{-1/2}(\C, \nu )$, $g\in W^{1,2}_{0}(\C, \nu )$, then $\langle Q^{-1/2}Df, Q^{1/2}Dg\rangle = \langle  Df,  Dg\rangle  $ (as an element of $L^1(\C, \nu)$). 
\end{itemize}
\end{Proposition}
\begin{proof}
 The proof of statement (i) is similar to  the standard proof in finite dimensions.  The mapping $u\mapsto Tu = (u, Q^{\theta}Du)$ is an isometry from $W^{1,p}_{\theta}(\C, \nu)$ to the product space $E:= L^p(\C, \nu) \times L^p(\C, \nu; X)$, which implies that  the range of $T$ is closed in $E$. Now, $L^p(\C, \nu)$ and $L^p(\C, \nu; X)$ are reflexive (e.g. \cite[Ch. IV]{DU})  so that $E$ is reflexive, and $T(W^{1,p}_{\theta}(\C, \nu))$ is reflexive too. Being isometric to a reflexive space,  $W^{1,p}_{\theta}(\C, \nu)$ is reflexive. 
 
Concerning (ii), the proof of the analogous statement for Gaussian measures given in \cite[Lemma 5.4.4]{Boga} works as well in our case. 
 
Statement (iii) is proved   approaching $f$ by a sequence of ${\mathcal F}{\mathcal C}^1_b(X)$ functions in $W^{1,2}_{-1/2}(X, \nu )$ (and hence, in $W^{1,2}_{0}(X, \nu )$) and 
approaching $g$ by a sequence of ${\mathcal F}{\mathcal C}^1_b(X)$ functions in $W^{1,2}_{0}(X, \nu )$ (and hence, in $W^{1,2}_{1/2}(X, \nu )$). Since the equality $\langle Q^{-1/2}Df_n, QDg_n\rangle = \langle  Df_n,  Dg_n\rangle  $ is true for the approximating functions, the claim follows letting $n\to \infty$.  
\end{proof}  
 
In the rest of the paper to simplify notation we shall drop the subindex $0$, namely we shall set $W^{1,p} (\C, \nu) :=  W^{1,p}_{0}(\C, \nu)$.

\subsection{Elliptic problems in the whole space, with regular $U$}
\label{ellittica}

Here we  report some results from \cite{DPL0} that will be used in the sequel. They concern weak solutions to 
\begin{equation}
\label{e1X}
\lambda u -  \mathcal K u = f, 
\end{equation}
where $ \mathcal K  u= \frac12\mbox{Tr}\;[D^2u] -\langle Ax + DU(x),Du\rangle$, in the case that  $U:X\mapsto \R$ is a differentiable convex function bounded from below, with Lipschitz continuous gradient. 

Given $\lambda >0$, $f\in L^2(X, \nu)$, a weak solution to \eqref{e1X} is a function $u\in W^{1,2}(X, \nu)$ such that 
\begin{equation}
 \label{weak}
\lambda\int_X u\,\varphi\,d\nu+\frac12\int_X\langle Du,D\varphi\rangle\,d\nu=\int_Xf\,\varphi\,d\nu,\quad \forall\varphi\in W^{1,2} (X,\nu).
 \end{equation}
Existence and uniqueness of a weak solution $u$ to  \eqref{e1X} is an easy consequence of the Lax--Milgram lemma. Taking $u$ as a test function and using the H\"older inequality in the right hand side we  obtain
\begin{equation}
\label{diss}
\lambda\int_X u^2d\nu +\frac12\int_X \|Du\|^2d\nu \leq 
 \frac{1}{\lambda} \int_X f^2d\nu  .
\end{equation}

We denote by  $K :D(K) \subset L^2(X, \nu)\mapsto L^2(X, \nu)$ the operator associated to the quadratic form  $(u, \varphi)\mapsto  \int_X\langle  Du, D\varphi\rangle\,d\nu$ in $W^{1,2} (X,\nu)$. So, the domain $D(K )$  consists of all $u\in W^{1,2} (X,\nu)$ such that there exists $v\in L^2(X, \nu)$ satisfying 
 $$\frac{1}{2}\int_X\langle  Du, D\varphi\rangle\,d\nu = - \int_Xv\,\varphi \,d\nu $$
 for all $\varphi\in  W^{1,2} (X,\nu)$, or equivalently for all $\varphi \in {\mathcal F}{\mathcal C}^1_b(X)$. 
In this case, $v =  Ku$. 
The weak solution $u$ to  \eqref{e1X} belongs to $D(K)$ and it is just $(\lambda I- K)^{-1}f$.

\begin{Theorem}
\label{stimeLip}
For every $\lambda >0$ and  $f\in L^2(X, \nu)$ the weak solution $u$  to \eqref{e1X} belongs to $W^{2,2}(X, \nu) $ $\cap$ $W^{1,2}_{-1/2}(X, \nu)$, and   the estimate
 \begin{equation}
\label{e25}
\ds\lambda\int_X \| Du \|^2d\nu +\frac12\int_X \quad\Tr\;[ (D^2u )^2]d\nu 
+\int_X\|Q^{-1/2 }Du\|^2d\nu  
\le 4\int_X f^2d\nu  
\end{equation}
holds. Moreover, the weak solution is also a strong solution in the Friedrichs  sense, that is: 
there is a sequence $(u_n)$ of ${\mathcal F}{\mathcal C}^2_b(X)$ functions (in fact, $u_n\in {\mathcal F}{\mathcal C}^3_b(X)$) that converge to $u$ in $L^2(X, \nu)$ and such that $\lambda u_n - {\mathcal K} u_n \to f$ in $L^2(X, \nu)$. 
\end{Theorem}

\begin{Remark}
\label{Rem:1}
{\em 
Note that estimates \eqref{diss} and \eqref{e25} imply that the above mentioned sequence of cylindrical functions $(u_n)$ converge to $u$ in $W^{2,2}(X, \nu) $ $\cap$ $W^{1,2}_{-1/2}(X, \nu)$. Indeed, it is sufficient to set $\lambda u_n - {\mathcal K} u_n = f_n$, and to use \eqref{diss} and \eqref{e25} with $u$ replaced by $u-u_n$ and $f$ replaced by $f-f_n$. }
 \end{Remark}


\section{The Neumann problem}
\label{sect_Neumann}


Throughout this section $U$ satisfies Hypothesis \ref{Hyp}. 
Moreover, $G:X\mapsto \R$   satisfies the assumptions of Lemma \ref{Le:chiusura} and $\C = G^{-1}((-\infty, 0])$ is a closed convex set. 
Given any $\lambda>0$ and $f\in L^2(\C, \nu)$, a function $u\in W^{1,2}(\C, \nu)$ is a weak solution to \eqref{Neumann} if 
\begin{equation}
\label{debole}
\lambda\int_{\C}  u\,\varphi\,d\nu+\frac12\int_{\C} \langle Du,D\varphi\rangle\,d\nu = \int_{\C}f\,\varphi\,d\nu,\quad \forall\varphi\in W^{1,2}(\C,\nu).
 \end{equation}
Since the restrictions to $\C$ of the functions in $ {\mathcal F}{\mathcal C}^1_b(X)$ are dense in  $W^{1,2}(\C,\nu)$, it is enough that the above equality is satisfied for  every $\varphi\in {\mathcal F}{\mathcal C}^1_b(X)$. 

Existence and uniqueness of a weak solution is an easy consequence of the Lax--Milgram Lemma. As in the case of the whole space, taking $\varphi= u$ in \eqref{debole} we obtain
\begin{equation}
\label{dissC}
\lambda\int_{\C} u^2d\nu +\frac12\int_{\C} \|Du\|^2d\nu \leq 
 \frac{1}{\lambda} \int_{\C} f^2d\nu  .
\end{equation}

\subsection{Regularity of weak solutions}
\label{Regularity}

Here we use the results of \S \ref{ellittica}  to study Sobolev regularity of the weak solutions to  \eqref{Neumann}.  
We approach the  problem in $\C$ by penalized problems in the whole space $X$, replacing $U$ by 
\begin{equation}
\label{Valpha}
V_{\alpha}(x):= U_{\alpha}(x)  + \frac{1}{2\alpha} \dist (x, \C)^2
\end{equation}
where $\alpha >0$ and $U_{\alpha}$ are  the approximations of $U$ given by Hypothesis \ref{Hyp}.

The corresponding Kolmogorov operator ${\mathcal K}_{\alpha} $ is defined on the smooth cylindrical functions by 
$${\mathcal K}_{\alpha}  u(x) = {{\mathcal L}}u(x) - \langle DU_{\alpha}(x), Du(x)\rangle - \frac{1}{\alpha} 
 \langle x-\Pi_{ \C}(x), Du(x)\rangle , \quad x\in X, $$
where $\Pi_{ \C}(x)$ is the unique element of $\C$ with minimal distance from $x$. $\Pi_{ \C}(x)$ is called ``projection of $x$ on $ \C$".

Since the function $x\mapsto DV_{\alpha}(x) = DU_{\alpha}(x)+ \frac{1}{\alpha} 
(x-\Pi_{\overline{\C}}(x))$ is Lipschitz continuous, the results of \S \ref{ellittica} may be applied. In particular, for every $\lambda >0$ and  $f\in C_b(X)$  the problem 
\begin{equation}
\label{Kappaalpha}
\lambda u_{\alpha} - {\mathcal K}_{\alpha}  u_{\alpha} = f
\end{equation}
has a unique weak solution $u_{\alpha}\in W^{1,2}(X, \nu_{\alpha})$, where
\begin{equation}
\label{nualpha}
\nu_{\alpha}(dx) :=   \exp(-2V_{\alpha}(x) )\mu(dx) .
\end{equation}
By Theorem  \ref{stimeLip}, $u_{\alpha}\in W^{2,2}(X, \nu_{\alpha})\subset W^{2,2}(X, \nu )$, and  estimate \eqref{e25} implies
\begin{equation}
\label{eq:a0}
 \begin{array}{l}
\ds\lambda\int_X \| Du_{\alpha} \|^2d\nu_{\alpha}  +\frac12\int_X \quad\Tr\;[ (D^2u_{\alpha} )^2]d\nu_{\alpha}  

\\
\\
 \ds +\int_X\|Q^{-1/2 }Du_{\alpha} \|^2d\nu_{\alpha}  +\int_X \langle D^2U_{\alpha}   Du_{\alpha} , Du_{\alpha}  \rangle d\nu_{\alpha}  \le 4\int_X f^2d\nu_{\alpha}   .
\end{array} 
\end{equation}
Taking into account that $U_{\alpha}\leq U$ and that $V_{\alpha} \geq C$ for each $\alpha$, from \eqref{diss} and \eqref{eq:a0} 
we obtain 
$$\|u_{\alpha}\|^2_{W^{2,2}(\C, \nu)} \leq \bigg( \frac{1}{\lambda^2} +  \frac{4}{\lambda} + 8\bigg) 
 \int_X f^2 d\nu_{\alpha} \leq \bigg( \frac{1}{\lambda^2} +  \frac{4}{\lambda} + 8\bigg)\|f\|_{\infty}^2e^{-2C}, $$
$$ \int_{\C} \|Q^{-1/2}Du_{\alpha}\|^2\,d\nu \leq 4\|f\|_{\infty}^2e^{-2 C}, $$
so that the restrictions of $u_{\alpha}$ to $\C$ are bounded in $W^{2,2}(\C, \nu)$ and in $W^{1,2}_{-1/2}(\C, \nu)$. A sequence $(u_{\alpha_{n|\C}})$ converges weakly to a function $u$ in $W^{2,2}(\C, \nu)$ and in $W^{1,2}_{-1/2}(\C, \nu)$.

\begin{Proposition}
\label{Pr:weak}
Let $f\in C_b(X)$ and let $\alpha_n \to 0$ be such that $(u_{\alpha_{n|\C}})$ converges weakly to a function $u$  in $W^{2,2}(\C, \nu)$ and in $W^{1,2}_{-1/2}(\C, \nu)$. 
Then $(u_{\alpha_{n|\C}})$ converges strongly  to   $u$ in $W^{1,2}(\C, \nu)$, moreover 
$u$ is the weak solution to \eqref{Neumann} and it satisfies
\begin{equation}
\label{stimaNeu}
\frac{1}{2} \int_{\C} \Tr\;[(D^2u )^2]\,d\nu +\int_{\C} \|Q^{-1/2}Du \|^2 d\nu 
\leq 4\|f\|_{L^2(\C, \nu)}^2 .
\end{equation}
\end{Proposition}
\begin{proof}
Since $\C $ is closed, $\dist (x, \C)>0$ for every $x\in \C^c$. Therefore, by dominated convergence, 
\begin{equation}
\label{convmis}
\limsup_{n\to\infty}  \int_{\C^c} e^{-2V_{\alpha_n}}d\mu \leq e^{-2C} \lim_{n\to \infty} \int_{\C^c} e^{-2\dist (x, \C)^2/{\alpha_n}}d\mu = 0.
\end{equation}

Let $\varphi\in {\mathcal F}{\mathcal C}^1_b(X)$, $n\in \N $. Then 
\begin{equation}
\label{debolen}
\lambda\int_{X}  u_{\alpha_n}\,\varphi\,d\nu_{\alpha_n}+\frac12\int_{X} \langle Du_{\alpha_n},D\varphi\rangle\,d\nu_{\alpha_n} = \int_{X}f\,\varphi\,d\nu_{\alpha_n} .
\end{equation}
The right hand side is splitted as the sum of an integral over $\C$ and an integral over $\C^c$. We have
$$\lim_{n\to \infty}  \int_{\C} f\,\varphi \,e^{-2U_{\alpha_n}}d\mu =  \int_{\C} f\,\varphi \,e^{-2U }d\mu $$
by dominated convergence, and  
$$\int_{\C^c} f\,\varphi \,e^{-2V_{\alpha_n}}d\mu \leq  \|f\|_{\infty}
\|\varphi\|_{\infty} \int_{\C^c} e^{-2V_{\alpha_n}}d\mu $$ 
that vanishes as $n\to \infty$, by \eqref{convmis}. So, the right hand side of \eqref{debolen} goes to $\int_{\C} f\,\varphi \,e^{-2U }d\mu  $  as $n \to \infty$. 

The integrals in the left hand side too are splitted as integrals over $\C$ and integrals over $\C^c$. Concerning the integrals over $\C$, arguing as in \cite[proof of Thm. 3.7]{DPL0} we obtain
$$\lim_{n\to \infty}  \int_{\C} (\lambda u_{\alpha_n}\,\varphi\ +\frac12  \langle Du_{\alpha_n},D\varphi\rangle) e^{-2U_{\alpha_n}}d\mu =  \int_{\C} (\lambda u \,\varphi\ +\frac12  \langle Du ,D\varphi\rangle) e^{-2U }d\mu .$$
Concerning  the integrals over $\C^c$, by the H\"older inequality we get
$$
\begin{array}{l}
\ds \bigg|\int_{\C ^c}  (\lambda u_{\alpha_n}\,\varphi\ +\frac12  \langle Du_{\alpha_n},D\varphi\rangle) e^{-2V_{\alpha_n}}d\mu \bigg|  
\\
\\
\ds \leq
\bigg( \int_{\C ^c}  (\lambda u_{\alpha_n}\,\varphi\ +\frac12  \langle Du_{\alpha_n},D\varphi\rangle)^2 e^{-2V_{\alpha_n}}d\mu\bigg)^{1/2 } \bigg( \int_{\C ^c} e^{-2V_{\alpha_n}}d\mu\bigg)^{1/2 }
\\
\\
\leq \ds 
  \|\varphi\|_{C^1_b(X)} (\|\lambda u_{\alpha_n}\|_{L^2(X, \nu_{\alpha_n})} + \frac{1}{2} 
\| \,\|Du_{\alpha_n}\|\, \|_{L^2(X, \nu_{\alpha_n})}) \bigg( \int_{\C ^c} e^{-2V_{\alpha_n}}d\mu\bigg)^{1/2 }
\end{array}$$
where $\|\lambda u_{\alpha_n}\|_{L^2(X, \nu_{\alpha_n})} $ and $\| \,\|Du_{\alpha_n}\|\, \|_{L^2(X, \nu_{\alpha_n})}$ are bounded by a constant independent of $n$ by \eqref{eq:a0}, and $ \int_{\C ^c} e^{-2V_{\alpha_n}}d\mu$ vanishes as $n\to \infty$ by \eqref{convmis}. 

Putting everything together and letting $n\to \infty$ in \eqref{debolen} we get 
$$\lambda\int_{\C}  u \,\varphi\,d\nu +\frac12\int_{\C} \langle Du ,D\varphi\rangle\,d\nu  = \int_{\C}f\,\varphi\,d\nu , $$
that is, $u$ is a weak solution to \eqref{Neumann}. Now, the argument of \cite[Lemma 3.8]{DPL0} shows that $u_{\alpha_{n|\C}}$ converges to $u$ in 
$W^{1,2}(\C, \nu)$. 

It remains to prove that $u$ satisfies \eqref{stimaNeu}. Since $u_{\alpha_{n|\C}}$ converges weakly to $u$ in $W^{2,2}(\C, \nu)$ and in $W^{1,2}_{-1/2}(\C, \nu)$, then
$$\frac{1}{2} \int_{\C} ( \Tr\;[(D^2u )^2]  + \|Q^{-1/2}Du \|^2) d\nu  
 \leq \limsup_{n\to \infty} 
\frac{1}{2} \int_{\C} ( \Tr\;[(D^2u_{\alpha_n})^2]  +  \|Q^{-1/2}Du_{\alpha_n} \|^2 )d\nu $$
$$\leq \limsup_{n\to \infty} 
\frac{1}{2} \int_{\C} ( \Tr\;[(D^2u_{\alpha_n})^2]  +  \|Q^{-1/2}Du_{\alpha_n} \|^2 ) d\nu_{\alpha_n} 
 \leq \limsup_{n\to \infty} 
4 \int_{X}f^2\,d\nu_{\alpha_n} .$$
Here we have used \eqref{e25} in the last inequality, and $U_{\alpha}\leq U$ in the last but one inequality. 

We already know that the integral $ \int_{X}f^2\,d\nu_{\alpha_n}$ goes to $ \int_{\C}f^2\,e^{-2U}d\mu$,   as $n\to \infty$. Then, \eqref{stimaNeu} follows. 
\end{proof}

\begin{Corollary}
\label{Cor:u}
For every $\lambda >0$ and $f\in L^2(\C, \nu)$ the weak solution $u$ of \eqref{Neumann} belongs to $W^{2,2}(\C, \nu)\cap W^{1,2}_{-1/2}(\C, \nu)$ and it satisfies \eqref{stimaNeu}. 
\end{Corollary}
\begin{proof}
Let $(f_n)$ be a sequence of functions in $C_b(X)$ that converge to the null extension of $f$ to $X$ in $L^2(X, \nu)$. By estimate \eqref{dissC}, the corresponding weak solutions to \eqref{Neumann} with $f$ replaced by $f_{n|\C}$ converge to $u$ in $W^{1,2}(\C, \nu)$, and by estimate \eqref{stimaNeu} they are a Cauchy sequence also in $W^{2,2}(\C, \nu)$ and in $W^{1,2}_{-1/2}(\C, \nu)$, so that $u\in W^{2,2}(\C, \nu)$ $\cap$ $W^{1,2}_{-1/2}(\C, \nu)$ and it satisfies \eqref{stimaNeu}. 
\end{proof}

\subsection{The Neumann boundary condition}
\label{Neu_bordo}

Here we show that the weak solution to  problem \eqref{Neumann} satisfies the Neumann condition $\langle Du, DG\rangle =0$ at $G^{-1}(0)$, in the sense of the traces of Sobolev functions. We need further assumptions on $G$.  
\begin{Hypothesis}
\label{HypG} 
$G:X\mapsto \R$ is a $C_{2,p}$-quasicontinuous function for every $p>1$, and
\begin{itemize}
\item[(i)] $G\in  \bigcap_{p>1} W^{2,p}_{1/2}(X, \mu), \quad {\ds \frac{1}{\|Q^{1/2}DG\|}} \in \bigcap_{p>1} L^p(X, \mu); $
\item[(ii)] $G\in  \bigcap_{p>1} W^{1,p}(X, \mu) .$
\end{itemize}
Moreover, $\C: = G^{-1}((-\infty, 0])$ is a closed convex nonempty set. 
\end{Hypothesis}

Let us recall that a function $G$  is called $C_{2,p}$-quasicontinuous if for each $\eps >0$ there is an open set $A\subset X$ such that $C_{2,p}(A)$  $\leq \eps$ and $G_{|X\setminus A}$ is continuous. Here $C_{2,p}(A)$ denotes the usual Gaussian capacity of order $(2,p)$, see \cite[Sect. 5.9]{Boga}. 
Every element  of $\cap_{p>1}W^{2,p}_{1/2}(X, \mu)$ has a version which is $C_{2,p}$-quasicontinuous   for every $p>1$.

Assumption (i) coincides with  the hypotheses of \cite{Tracce}, where it was shown that the elements of $W^{1,q}_{1/2} (\mathcal C, \mu)$ with $q>1$ have traces at $G^{-1}(0)$, as well as at the other level sets of $G$. Such traces belong to $L^1(G^{-1}(0), \rho)$, where $\rho$ is the Hausdorff-Gauss measure of Feyel and De la Pradelle  \cite{FP}. 

Assumption (ii) will be used in Proposition  \ref{finale},  for the Neumann  condition $\langle Du, DG\rangle$ $ =$ $0$ be meaningful. To this aim one needs that $DG$ exists (at least, near $G^{-1}(0)$), and (i) is not sufficient.

The trace of any Sobolev function $u\in  W^{1,p}_{1/2} (\mathcal C, \mu)$ at $G^{-1}(0)$ is  denoted by $u_{|G^{-1}(0)}$.  It coincides $\rho$-a.e. with any $C_{1,p}$-quasicontinuous version of $u$, in particular if $u$ has a continuous version, its trace coincides $\rho$-a.e. with the restriction of $u$ to $G^{-1}(0)$. The trace operator $u\mapsto u_{|G^{-1}(0)}$ is bounded from $ W^{1,p}_{1/2} (X, \mu)$  to $L^q(G^{-1}(0)), \rho)$, for every $q\in [1, p)$. (Under further assumptions that are not needed here, it is bounded from $ W^{1,p}_{1/2} (X, \mu)$  to $L^p(G^{-1}(0)), \rho)$). 
The integration formula 
\begin{equation}
\label{parti}
\int_{\C}  D_ku  \,d\mu = \int_{\C} \frac{x_k}{\lambda_k}u\,d\mu + \int_{G^{-1}(0)} u_{|G^{-1}(0)}  \frac{D_kG}{\|Q^{1/2}DG\|} \,d\rho 
\end{equation}
holds for every $k\in \N$ and $u\in  W^{1,p}_{1/2} (X, \mu)$ for some $p>1$.  For the proofs of these statements and for other properties of traces we refer to  \cite{Tracce}. 

Moreover, if $u (x)\geq C$ for $\mu$-a.e. $x$, then $u_{|G^{-1}(0)}\geq C$, $\rho$-a.e. This is not immediate, since under Hypothesis \ref{HypG}, the set $G^{-1}(0)$ is $\mu$-negligible. However, this 
can be seen approaching $u$ in $W^{1,p}_{1/2} (\mathcal C, \mu)$ by a sequence of continuous functions $u_n \in W^{1,p}_{1/2} (\mathcal C, \mu)$ such that $u_n(x)\geq C$ for every $x\in C$ (for instance, we may take $u_n(x) = \max\{v_n(x), \,C\}$, where $(v_n)$ is any sequence of Lipschitz continuous functions approaching $u$ in $W^{1,p}_{1/2} (\mathcal C, \mu)$). Since $u_n(x)\geq C$ for every $x\in G^{-1}(0)$, and $u_{n|G^{-1}(0)}$ converges to $u_{|G^{-1}(0)}$ in $L^1(G^{-1}(0), \rho)$, the statement follows.

We shall show that the elements of $W^{1,p} (\C, \nu)$ with $p>p_0/(p_0-1) $ have traces at the boundary that belong to $L^1(G^{-1}(0), e^{-2U}d\rho)$. 
The weight $e^{-2U}$ is meaningful in the surface integrals, as the next lemma shows. 

\begin{Lemma}
\label{Le:U}
$U(x)<\infty$ for $\rho$-a.e. $x\in G^{-1}(0)$, and the trace of $\exp(-2U)$ coincides with $\exp(-2U_{|G^{-1}(0)})$ for $\rho$-a.e. $x\in G^{-1}(0)$. 
\end{Lemma}
\begin{proof}
Since  $U\in W^{1,p_0}_{1/2} (X, \mu)$,  its trace at $G^{-1}(0)$ belongs to $L^1(G^{-1}(0), \rho)$ hence it cannot be equal to $+\infty$ in a set with positive measure. Let us show that  $ e^{-2U}_{|G^{-1}(0)}$ coincides with $ e^{-2U_{|G^{-1}(0)}}$ for $\rho$-a.e. $x\in G^{-1}(0)$.  
We already know that  $U_{\alpha}$ and $e^{-2U_{\alpha}}$ converge to $U$ and to $e^{-2U }$ in $W^{1,p_0}(X, \mu)$ as $\alpha \to 0$. 
Then, their traces at $G^{-1}(0)$ converge to   $U_{|G^{-1}(0)}$ and  of $(e^{-2U })_{|G^{-1}(0)}$ in $L^1(G^{-1}(0), \rho)$. 
$U_{\alpha}$ and $e^{-2U_{\alpha}}$ are continuous, their traces are just their restrictions at $G^{-1}(0)$, $\rho$-a.e. 
It remains to show that $e^{-2U_{\alpha |G^{-1}(0)}}$ converges to $e^{-2U_{  |G^{-1}(0)}}$ in $L^1(G^{-1}(0), \rho)$. To this aim we remark that 
$U_{|G^{-1}(0)}\geq C$, $\rho$-a.e.
Now, since $\xi\mapsto e^{-2\xi}$ is Lipschitz continuous in $[C, +\infty)$ and both $U_{\alpha |G^{-1}(0)}$ and $U_{|G^{-1}(0)}$ have values in $[C, +\infty)$, then $e^{-2U_{\alpha |G^{-1}(0)}}$ converges to $e^{-2U_{  |G^{-1}(0)}}$ in $L^1(G^{-1}(0), \rho)$, along the converging subsequence, and the statement follows. \end{proof}

\vspace{3mm}

As a first step we establish a formula similar to \eqref{parti}, that involves  the measure $\nu$ in $\C$ and the measure $e^{-2U}d\rho$ in $G^{-1}(0)$. 
 
\begin{Lemma}
\label{Le:partinu}
Let $u\in  W^{1,p}_{1/2}(\C, \nu)$, with $p> p_0/(p_0-1) $. Then for every $k\in \N$ we have
\begin{equation}
\label{partinu}
\int_{\C}  D_ku  \,e^{-2U }d\mu = \int_{\C}(2D_kU  + \frac{x_k}{\lambda_k})u\,e^{-2U }d\mu + \int_{G^{-1}(0)}  u_{ |G^{-1}(0)}  \frac{D_kG}{\|Q^{1/2}DG\|}\,e^{-2U_{|G^{-1}(0)}}d\rho .
\end{equation}
\end{Lemma}
\begin{proof} It is sufficient to apply formula \eqref{parti} to the function $u   e^{-2U }$, which  belongs to $W^{1,r}_{1/2}(X, \mu)$ with  $r=p_0p/(p_0+p)$, and to remark that the trace of the product $u   e^{-2U }$ is the product of the respective traces. 
\end{proof}

\begin{Proposition}
\label{Pr:traccia}
Let $u\in C^1_b(X)$ and let $p > p_0/(p_0-1) $. Then for every $q\in [1, p(p_0-1)/p_0 )$
there is $C_0= C_0(G, q)>0$, independent of $u$ and $U$, such that  
\begin{equation}
\label{maggtraccia}
\int_{G^{-1}(0)} |u|^q e^{-2U_{|G^{-1}(0)}}d\rho \leq C_0  \|u\|_{W^{1,p}_{1/2}(\C,\nu)}^q e^{-2C (p-q)/p}(1+\|Q^{1/2} DU\|\,\|_{L^{p_0}(X, \mu)}).
\end{equation}
\end{Proposition}
\begin{proof}
Applying   \eqref{partinu} to the functions $\lambda_k |u|^q D_kG$, that belong to $W^{1,s}_{1/2}(X, \mu)$ for every $s>1$, and summing over $k$, we get
$$\int_{\C} ( q|u|^{q-2}u\langle Q^{1/2}Du, Q^{1/2}DG\rangle +  (L_0G -2\langle Q^{1/2}DU, Q^{1/2}DG\rangle)|u|^q)
  \,e^{-2U }d\mu $$
$$=   \int_{G^{-1}(0)} |u|^q  \|Q^{1/2}DG\| \,e^{-2U }d\rho , $$
where $L_0$ is the Ornstein--Uhlenbeck operator of the Malliavin Calculus, 
$$L_0G(x) = \sum_{k=1}^{\infty} (\lambda_k D_{kk}G(x) - x_kD_kG(x)). $$
Here, for every $s>1$ the series converges in $L^s(X, \mu)$ and there is $c_s>0$ such that $\|L_0G\|_{L^s(X, \mu)}\leq c_s\|G\|_{W^{2,s}_{1/2}(X, \mu)}$ (e.g., \cite[\S 5.8]{Boga}). 
Then, for every $q\geq 1$, 
$$ \int_{G^{-1}(0)} |u|^q  \|Q^{1/2}DG\| \,e^{-2U_{|G^{-1}(0)} }d\rho \leq $$
$$\leq \int_{\C} ( q|u|^{q-1} \|Q^{1/2}Du\|\,  \|Q^{1/2}DG\| + |u|^q(|L_0G| + 2\|Q^{1/2}DU\|\, \|Q^{1/2}DG\|))e^{-2U} d\mu.$$
The latter integral is finite, since $u$ and $\|Q^{1/2}Du\|$ are bounded,  
$ \|Q^{1/2}DG\|$, $L_0G \in L^s(\C, \mu) \subset L^s(\C, \nu)$ for every $s$, and $\|Q^{1/2}DU\|\in L^{p_0}(\C, \mu) \subset L^{p_0}(\C, \nu)$. Using  the H\"older inequality to estimate it in terms of $ \|u\|_{W^{1,p}_{1/2}(\C, \nu)}$, with $p>q$, we get
$$\begin{array}{ll}
(i) &
 \ds \int_{\C} q |u|^{q-1} \|Q^{1/2}Du\|\,  \|Q^{1/2}DG\| e^{-2U} d\mu
 \\
 \\
&  \leq q\|u\|_{L^p(\C, \nu)}^{q-1} \|\,\|Q^{1/2}Du\|\,\|_{L^p(\C, \nu)} \|\,  \|Q^{1/2}DG\| \|_{L^{p/(p-q)}(\C, \nu)} ,
\\
\\
(ii) & \ds
\int_{\C}  |u|^q  |L_0G|\,e^{-2U} d\mu 
\leq \|u\|_{L^p(\C, \nu)}^{q} \|\,  L_0G\|_{L^{p/(p-q)}(\C, \nu)},
\\
\\
(iii) & \ds  \int_{\C}    |u|^q \|Q^{1/2}DU\|\, \|Q^{1/2}DG\|\,e^{-2U} d\mu
\leq   \|u\|_{L^p(\C, \nu)}^{q} \|\,\|Q^{1/2}DU\| \,\|_{L^{p_0}(\C, \nu)}
 \|\,\|Q^{1/2}DG\| \,\|_{L^{s}(\C, \nu)} ,
 \end{array}$$
with $1/s= 1-(q/p + 1/p_0)$. Then, 
\begin{equation}
\label{conquibus}
\int_{G^{-1}(0)} |u|^q  \|Q^{1/2}DG\| \,e^{-2U_{|G^{-1}(0)} }d\rho
 \leq C_1 \|u\|_{W^{1,p}_{1/2}(\C,\nu)}^q  e^{-2C (p-q)/p}(1+\|Q^{1/2} DU\|\,\|_{L^{p_0}(X, \mu)}), 
 \end{equation}
where $C_1>0$ depends only on $q$ and $G$. 

Now we recall  that $ \|Q^{1/2}DG\|^{-1}\in L^s (G^{-1}(0),  \rho) 
 \subset L^s (G^{-1}(0), e^{-2U_{|G^{-1}(0)}} \rho)$ $ \forall s>1$. 
Going back to  \eqref{maggtraccia}, 
for any $u\in C^1_b(X)$ and $1\leq q < r < p(p_0-1)/p_0 $ we may write $\int_{G^{-1}(0)} |u|^q  e^{-2U }d\rho = 
\int_{G^{-1}(0)} |u|^q  \|Q^{1/2}DG\|^{q/r} \, \|Q^{1/2}DG\|^{-q/r} e^{-2U_{|G^{-1}(0)} }d\rho$, and using the H\"older inequality, estimate \eqref{conquibus} with $r$ instead of $q$,  we obtain \eqref{maggtraccia}.  
\end{proof}

Let $v\in W^{1,p}_{1/2}(\C, \nu)$ for some $p > p_0/(p_0-1) $. By proposition \ref{Pr:traccia},  for every sequence $(v_n)\subset C^1_b(X)$ such that $v_{n|\C}$ converges to $v$ in $W^{1,p}_{1/2}(\C, \nu)$,   $v_{n|G^{-1}(0)}$ is a Cauchy sequence in 
$L^q(G^{-1}(0), e^{-2U}d\rho)$, for each $q\in [1, p(p_0-1)/p_0 )$. This allows to define the traces at the boundary of the elements of $W^{1,p}_{1/2}(\C, \nu)$. 

\begin{Definition}
Let $v\in W^{1,p}_{1/2}(\C, \nu)$ for some $p>p_0/(p_0-1) $. The trace of $v$ at $G^{-1}(0)$ is the unique element of $L^1(G^{-1}(0), e^{-2U_{|G^{-1}(0)}} \rho)$ to which $v_{n|G^{-1}(0)}$ converges, for every  sequence $(v_n)\subset C^1_b(X)$ such that $v_{n|\C}$ converges to $v$ in $W^{1,p}_{1/2}(\C, \nu)$. It is denoted by 
$v_{|G^{-1}(0)}$. 
\end{Definition}

\begin{Corollary}
\label{Cor:traccia}
The trace at $G^{-1}(0)$ of any $v\in W^{1,p}_{1/2}(\C, \nu)$ belongs to $L^q(G^{-1}(0), e^{-2U_{|G^{-1}(0)}} \rho)$ for $1\leq q<p(p_0-1)/p_0 $, and estimate
\eqref{maggtraccia} holds with $u$ replaced by $u_{|G^{-1}(0)}$ in the surface integral. Therefore,   the mapping $W^{1,p}_{1/2}(\C, \nu)\mapsto L^q(G^{-1}(0), e^{-2U}d\rho)$, $v\mapsto v_{|G^{-1}(0)}$ is bounded. 
Since   $W^{1,p} (\C, \nu) \subset W^{1,p}_{1/2}(\C, \nu)$ with continuous embedding, the same holds for functions in $W^{1,p} (\C, \nu)$. 
\end{Corollary}

Let $u$ be the weak solution $u$ to \eqref{Neumann}. The main result of this section is the fact that    $\langle Du, DG\rangle  $ has   trace at $G^{-1}(0)$, and  that such trace vanishes.  

The first step is the following proposition. 

\begin{Proposition}
\label{difficile}
Let $f\in C_b(X)$ and let $u$, $u_{\alpha}$  be the weak solutions   to \eqref{Neumann} and to \eqref{Kappaalpha}, respectively. Then 
\begin{itemize}
\item[(i)] For every $\alpha >0$ and $p<2$, $\langle Du_{\alpha}, DG\rangle \in W^{1,p}_{1/2}(X, \nu_{\alpha})$, and there is $C_p$ independent of $f$ and $\alpha$ such that $\|\langle Du_{\alpha}, DG\rangle\|_{ W^{1,p}_{1/2}(X, \nu_{\alpha})}\leq C_p\|f\|_{L^2(X, \nu_{\alpha})}$. 
\item[(ii)] The function $\langle Du, DG\rangle  $ belongs to $W^{1,p}_{1/2}(\C, \nu)$, for every $p\in [p_0', 2)$, and  $\|\langle Du , DG\rangle\|_{ W^{1,p}_{1/2}(\C, \nu )}$ $\leq$ $ C_p\|f\|_{L^2(X, \nu )}$, with the same constant $C_p$ as in (i). 
\item[(iii)] There is a vanishing sequence $(\alpha_n)$ such that $u_{\alpha_n|\C}$ converge weakly to $u$ in $W^{2,2}(\C, \nu)$ and in  $W^{1,2}_{-1/2}(\C, \nu)$,   $\langle Du_{\alpha_n}, DG\rangle _{|\C} $ converges weakly to $\langle Du, DG\rangle_{|\C} $ in $W^{1,p}_{1/2}(\C, \nu)$. 
\end{itemize}
\end{Proposition}
\begin{proof}
 By Remark \ref{Rem:1}, there exists a sequence  $(u_{\alpha, n} )\subset {\mathcal F}{\mathcal C}^2_b(X) $ that converges to $u_{\alpha}$ in $W^{2,2} (X, \nu_{\alpha})$ $\cap$ $ W^{1,2}_{-1/2}(X, \nu_{\alpha})$, and such that $f_n:= \lambda u_{\alpha, n}  - K_{\alpha_n}u_{\alpha, n} $ goes to $f$ in $L^2(X, \nu_{\alpha})$ as $n\to \infty$. We set
 $$v_{\alpha} :=\langle Du_{\alpha}, DG\rangle  , \quad v_{\alpha, n} := \langle Du_{\alpha, n}, DG\rangle  . $$
Then,  $v_{\alpha, n} = \sum_{k=1}^{\infty} D_ku_{\alpha, n}D_kG \in W^{1,p}_{1/2}(X, \nu_{\alpha})$ for every $p>1$, since the series is in fact a finite sum (note that, since $DV_{\alpha}$ is Lipschitz continuous, then $\|DV_{\alpha}\|\in L^p(X, \nu)$ for every $\alpha$, and the Sobolev spaces with respect to $\nu_{\alpha}$ are well defined for every $p>1$). By the H\"older inequality,  $\lim_{n\to \infty}v_{\alpha, n}  = v_{\alpha} $ in $L^p(X, \nu_{\alpha})$. Possibly  replacing $v_{\alpha, n}  $ by a subsequence, we may assume that $(v_{\alpha, n}  )$ converges to $ v_{\alpha} $, $\nu_{\alpha}$-a.e.
 
Now we prove that the sequence $(v_{\alpha, n} )$ is bounded in $W^{1,p}_{1/2}(X, \nu_{\alpha})$, for $p<2$. 
 For every $k\in \N$ we have
$$\begin{array}{l}
| D_k \langle Du_{\alpha, n}, DG\rangle |   \leq   \ds \bigg|\sum_{j\in \N} (D_{kj}u_{\alpha, n}D_jG + D_juD_{kj}G)\bigg|
\\
\\
  \leq  \ds \bigg(\sum_{j\in \N} (D_{kj}u_{\alpha, n})^2\bigg)^{1/2} \bigg(\sum_{j\in \N} (D_{j}G)^2\bigg)^{1/2}
 + \bigg(\sum_{j\in \N}\lambda _j^{-1}(D_{j}u_{\alpha, n})^2\bigg)^{1/2}\bigg(\sum_{j\in \N}\lambda _j (D_{kj}G)^2\bigg)^{1/2}
\end{array}$$
so that 
$$\begin{array}{l}
\|Q^{1/2}D(\langle Du_{\alpha, n}, DG\rangle)\| 
\\
\\
\leq \sqrt{2}[ \max_{k\in \N} \lambda_k^{1/2} ( \textrm{Tr}(D^2u_{\alpha, n}^2))^{1/2}\|DG\| + \|Q^{-1/2}Du_{\alpha, n}\|  (\textrm{Tr}(Q^{1/2}D^2GQ^{1/2})^2 )^{1/2}].
\end{array}$$
By our assumptions,    $\|DG\| $ and    $\textrm{Tr}(Q^{1/2}D^2GQ^{1/2})^2 )^{1/2}$ belong to $L^s(X, \mu)\subset L^s(X,  \nu_{\alpha})$ for every $s$. 
Since $e^{-2U_{\alpha}}\leq e^{-2C}$, their $L^s(X,  \nu_{\alpha})$-norm is bounded by a constant independent of $\alpha$. 
Using the H\"older inequality  with with $s= 2/(2-p)$,  we obtain
$$
\int_{X}\|Q^{1/2}D(\langle Du_{\alpha, n}, DG\rangle)\|^p d\nu_{\alpha} \leq c_p(\|u_{\alpha, n}\|_{W^{2,2}(X, \nu_{\alpha})} +  \|u_{\alpha, n}\|_{W^{1,2}_{-1/2}(X, \nu_{\alpha})})^{p }, 
$$
where $c_p>0$ does not depend on $\alpha$ and $n$. By estimates \eqref{diss} and \eqref{e25}, 
$$\|u_{\alpha, n}\|_{W^{2,2}(X, \nu_{\alpha})} +  \|u_{\alpha, n}\|_{W^{1,2}_{-1/2}(X, \nu_{\alpha})} \leq k_{\lambda}\|f_n\|_{L^{2,}(X, \nu_{\alpha})}, $$ 
where $k_{\lambda}$ depends only on $\lambda$. Therefore, 
$$\limsup_{n\to \infty}\|\langle Du_{\alpha, n}, DG\rangle)\|_{W^{1,p}_{1/2}(X, \nu_{\alpha})} \leq C_{p, \lambda}\|f\|_{L^{2}(X, \nu_{\alpha})} .$$  
where $C_{p, \lambda}$ is independent of $f$,  $\alpha $ and $n$. Applying now Proposition \ref{varie}(ii), with  $X$ replacing $\C$  and $\nu_{\alpha}$  replacing $\nu$,  yields   statement (i). 

 Let now $u_{\alpha_n}$ be any sequence of solutions to \eqref{Kappaalpha} such that the restrictions $u_{\alpha_n|\C}$ converge weakly to $u$ in $W^{2,2}(\C, \nu)$ and in  $W^{1,2}_{-1/2}(\C, \nu)$. 
Since $W^{1,p}_{1/2}(X, \nu_{\alpha})\subset W^{1,p}_{1/2}(X, \nu)$
  for every $p\in [p_0', 2)$, by (i) the sequence $\langle Du_{\alpha_n}, DG\rangle _{|\C}$ is bounded in $W^{1,p}_{1/2}(\C, \nu)$. More precisely, we have 
 $$\|\langle Du_{\alpha_n}, DG\rangle _{|\C}\|_{W^{1,p}_{1/2}(\C, \nu)} \leq  \|\langle Du_{\alpha_n}, DG\rangle _{|\C}\|_{W^{1,p}_{1/2}(\C, \nu_{\alpha_n})} \leq  C_{p, \lambda} \|f\|_{L^2(X, \nu_{\alpha_n})},$$
 so that 
\begin{equation}
\label{civuole}\limsup_{n\to \infty} \|\langle Du_{\alpha_n}, DG\rangle _{|\C}\|_{W^{1,p}_{1/2}(\C, \nu)} \leq \limsup_{n\to \infty} C_{p, \lambda} \|f\|_{L^2(X, \nu_{\alpha_n})} = C_{p, \lambda} \|f\|_{L^2(X, \nu )}. 
\end{equation}
Since $W^{1,p}_{1/2}(\C, \nu)$ is reflexive by Proposition \ref{varie}(i), a subsequence of $\langle Du_{\alpha_n}, DG\rangle  _{|\C}$ converges weakly in $W^{1,p}_{1/2}(\C, \nu)$ to an element $\psi \in W^{1,p}_{1/2}(\C, \nu)$, that satisfies $\|\psi\|_{W^{1,p}_{1/2}(\C, \nu)} \leq C_{p, \lambda} \|f\|_{L^2(X, \nu )}$, by 
\eqref{civuole}. 

Let us identify $\psi $ with $\langle Du, DG\rangle  $. Indeed, by the H\"older inequality, 
the mapping $v\mapsto \langle Dv, DG\rangle$ is bounded from $W^{1,2}(\C, \nu)$ to $L^p(\C, \nu)$. Since  $u_{\alpha_{n}|\C}$   converges weakly to $u$ in $W^{1,2}(\C, \nu)$, then $\langle Du_{\alpha_n}, DG\rangle  _{|\C}$ converges weakly to $\langle Du, DG\rangle  $ in $L^{p} (\C, \nu)$. Then 
$\psi =\langle Du, DG\rangle \in L^{p} (\C, \nu)$. This proves statements (ii) and (iii). 
\end{proof}

The reason why we need two steps in the proof of Proposition \ref{difficile} is that, while the sequence of  cylindrical functions $(u_{\alpha, n})$ that approaches $u_{\alpha}$ is bounded   both in $W^{2,2}(X, \nu)$ and in $W^{1,2}_{-1/2}(X, \nu)$, it seems not easy to find a sequence of cylindrical functions that approaches $u$ and that is bounded  both in $W^{2,2}(\C, \nu)$ and in $W^{1,2}_{-1/2}(\C, \nu)$.

As a consequence of Proposition \ref{difficile}(ii), the function $ \langle Du, DG\rangle $ has trace at $G^{-1}(0)$, that belongs to $L^q(G^{-1}(0), \rho)$ for every $q<2$. To show that such a trace vanishes we shall use 
the integration formula of the next lemma for the approximating functions $u_{\alpha_n}$. 

\begin{Lemma}
\label{Le:ualpha}
Fix $\alpha >0$, $f\in L^2(X, \nu_{\alpha})$, and let $u_{\alpha}$ be the weak solution to 
\eqref{Kappaalpha}. Then, for every 
$\varphi \in C^1_b(X)$ we have
\begin{equation}
\label{eq:ualpha}
  \int_{\C}  K_{\alpha}u_{\alpha}\varphi \,e^{-2U_{\alpha}}d\mu =  - \frac{1}{2} \int_{\C} \langle Du_{\alpha}, D\varphi\rangle e^{-2U_{\alpha}}d\mu 
+  \frac{1}{2} \int_{G^{-1}(0)} \varphi\, \frac{\langle Du_{\alpha}, DG\rangle }{\|Q^{1/2}DG\|}\,e^{-2U_{\alpha}}d\rho .
\end{equation}
\end{Lemma}
\begin{proof} By Theorem \ref{stimeLip} and Remark  \ref{Rem:1},  there exists a sequence  $u_{\alpha, n}\in  {\mathcal F}{\mathcal C}^2_b(X)$   such that $\lim_{n\to \infty} u_{\alpha, n}$ $ = u_{\alpha}$, 
$\lim_{n\to \infty} {\mathcal K}_{\alpha}u_{\alpha, n} =  K_{\alpha}u_{\alpha} =   \lambda u_{\alpha}- f$, in $L^2(X, \nu_{\alpha})$, and $\lim_{n\to \infty} u_{\alpha, n} = u_{\alpha}$ in $W^{2,2}_{0}(X, \nu_{\alpha})$ $\cap$ $ W^{1,2}_{-1/2}(X, \nu_{\alpha})$. 
For every $k\in \N$,  $\varphi D_k u_{\alpha, n}\in C^1_b(X)$. 
Replacing $u$ by $\varphi D_k u_{\alpha, n}$  in \eqref{partinu}, with $\nu$ replaced by  $\nu_{\alpha}$, and summing over $k$ yields 
\begin{equation}
\label{3.15}
 \int_{\C}  {\mathcal K}_{\alpha} u_{\alpha, n}\,\varphi \,e^{-2U_{\alpha}}d\mu =  - \frac{1}{2} \int_{\C} \langle Du_{\alpha, n}, D\varphi\rangle e^{-2U_{\alpha}}d\mu 
 +  \frac{1}{2} \int_{G^{-1}(0)} \varphi\, \frac{\langle Du_{\alpha, n}, DG\rangle }{\|Q^{1/2}DG\|}\,e^{-2U_{\alpha}}d\rho .
\end{equation}
The integrals over $\C$ converge to their respective limits by dominated convergence. Concerning the integral over $G^{-1}(0)$, in the proof of Proposition \ref{difficile} we have shown that the sequence $(\langle Du_{\alpha, n}, DG\rangle)$ is bounded in $W^{1,p}_{1/2}(X, \nu_{\alpha})$ and converges to $\langle Du_{\alpha}, DG\rangle$ in   $L^p(X, \nu_{\alpha})$. Since $W^{1,p}_{1/2}(X, \nu_{\alpha})$ is reflexive, a subsequence converges weakly to $\langle Du_{\alpha}, DG\rangle$. The linear functional
$$v\mapsto \int_{G^{-1}(0)} \varphi\, \frac{\langle Dv, DG\rangle }{\|Q^{1/2}DG\|}\,e^{-2U_{\alpha}}d\rho $$
is bounded in  $W^{1,p}_{1/2}(X, \nu_{\alpha})$, hence it is an element of $(W^{1,p}_{1/2}(X, \nu_{\alpha}))'$. 
Letting $n\to \infty$ along the weakly convergent subsequence, yields 
$$\lim_{n\to \infty}  \int_{G^{-1}(0)} \varphi\, \frac{\langle Du_{\alpha, n}, DG\rangle }{\|Q^{1/2}DG\|}\,e^{-2U_{\alpha}}d\rho = \int_{G^{-1}(0)} \varphi\, \frac{\langle Du_{\alpha}, DG\rangle }{\|Q^{1/2}DG\|}\,e^{-2U_{\alpha}}d\rho , $$
and \eqref{eq:ualpha} follows.

Note that the approximation procedure   is needed, because we do not know whether the series $\sum_k (D_{kk}u_{\alpha}/2 - x_kD_ku_{\alpha}/\lambda_k + D_kUD_ku_{\alpha})$ converges to $ {\mathcal K}_{\alpha}  u_{\alpha}$
in $L^1(\C, \nu_{\alpha})$, while replacing $u_{\alpha}$ by  $u_{\alpha, n}$ this is just a finite sum, and \eqref{3.15} follows. 
\end{proof}

\begin{Theorem}
\label{finale}
For  $\lambda >0$ and $f\in L^2(\C, \nu)$, let $u$ be the weak solution  of \eqref{Neumann}. 
Then $\langle Du, DG\rangle =0$ at $G^{-1}(0)$, $\rho$--a.e.
\end{Theorem}
\begin{proof} To begin with, we consider data $f\in C_b(X)$. Using a subsequence of the approximating functions $u_{\alpha}$ and Lemma  \ref{Le:ualpha}, we shall prove that 
\begin{equation}
\label{intbordo}
  \int_{G^{-1}(0)}\varphi \, \frac{\langle Du , DG\rangle}{\|Q^{1/2}DG\|}\,  e^{-2U_{|G^{-1}(0)} }d\rho = 0, \quad \varphi \in C^1_b(X).
\end{equation}
Then, 
using Proposition \ref{difficile} we will prove that \eqref{intbordo} holds even if $f\in L^2(\C, \mu)$. 
From \eqref{intbordo} the statement will follow.

\vspace{3mm}
\noindent{\em First step: if $f\in C_b(X)$, \eqref{intbordo} holds.}
For every $\alpha >0$ let $u_{\alpha }$ be the weak solution to 
\eqref{Kappaalpha}. Fix $p\in (p'_0, 2)$, and let
 $(\alpha_n)$ be any vanishing sequence such that  
$(u_{\alpha_n|\C})$ converges weakly to $u$ in $W^{2,2}(\C, \nu)$ and in $W^{1,2}_{-1/2}(\C, \nu)$, and 
  $\langle Du_{\alpha_n}, DG\rangle _{|\C} $ converges weakly to $\langle Du, DG\rangle  $ in $W^{1,p}_{1/2}(\C, \nu)$. Such sequence exists by Proposition \ref{difficile}. Moreover, possibly choosing a further subsequence, we may assume that $\exp(-2U_{\alpha_n}(x))\to \exp(-2U_{|G^{-1}(0)}(x)$, $\rho$-a.e. in $G^{-1}(0)$. Indeed,   as we already remarked in the proof of Lemma \ref{Le:U}, since $\exp(-2U_{\alpha_n})$ converges to $\exp(-2U )$ in $W^{1,p_0}_{1/2}(\C,\mu)$, the trace of $\exp(-2U_{\alpha_n})$ at $G^{-1}(0)$ converges to the trace of $\exp(-2U )$ in $L^1(G^{-1}(0), \rho)$. 

By Lemma  \ref{Le:ualpha}, for every $\varphi \in C^1_b(X)$  and $n\in \N$ we have
\begin{equation}
\label{eq:ualphan}
\int_{\C} (\lambda u_{\alpha_n}- f)\varphi \,e^{-2U_{\alpha_n}}d\mu =    - \frac{1}{2} \int_{\C} \langle Du_{\alpha_n}, D\varphi\rangle e^{-2U_{\alpha_n}}d\mu 
 +  \frac{1}{2} \int_{G^{-1}(0)} \varphi\, \frac{\langle Du_{\alpha_n}, DG\rangle _{|G^{-1}(0)}}{\|Q^{1/2}DG\|}\,e^{-2U_{\alpha_n}}d\rho .
\end{equation}
Letting $n\to \infty$,  the proof of Proposition  
\ref{Pr:weak} yields
$$\lim_{n\to \infty}\int_{\C} (\lambda u_{\alpha_n}- f)\varphi \,e^{-2U_{\alpha_n}}d\mu 
  =  \int_{\C} (\lambda u -f)\,\varphi \,e^{-2U }d\mu , $$
$$\lim_{n\to \infty}\frac{1}{2} \int_{\C} \langle Du_{\alpha_n}, D\varphi\rangle e^{-2U_{\alpha_n}}d\mu  
 = \frac{1}{2} \int_{\C} \langle Du , D\varphi\rangle e^{-2U }d\mu .$$
We split the surface integral in \eqref{eq:ualphan} as $I_{1,n} + I_{2,n}$, where
$$I_{1,n}  =  \int_{G^{-1}(0)} \varphi\, \frac{\langle Du_{\alpha_n} , DG\rangle _{|G^{-1}(0)}}{\|Q^{1/2}DG\|}\,e^{-2U_{|G^{-1}(0)} }d\rho  $$
$$I_{2,n}=  \int_{G^{-1}(0)} \varphi\, \frac{\langle Du_{\alpha_n} , DG\rangle _{|G^{-1}(0)}}{\|Q^{1/2}DG\|}\,(e^{-2U_{\alpha_n} }- e^{-2U_{|G^{-1}(0)}})d\rho  .$$
Since $\varphi \exp(-2U_{|G^{-1}(0)}) \in L^q(G^{-1}(0), \rho)$ for every $q>1$, the mapping
\begin{equation}
\label{dual}
v\mapsto \int_{G^{-1}(0)} \varphi\, \frac{\langle Dv , DG\rangle _{|G^{-1}(0)}}{\|Q^{1/2}DG\|}\,e^{-2U_{|G^{-1}(0)}}d\rho
\end{equation}
is in the dual space of $W^{1,p}_{1/2}(\C, \nu)$. Since $\langle Du_{\alpha_n}, DG\rangle _{|\C} $ converges weakly to $\langle Du, DG\rangle  $ in $W^{1,p}_{1/2}(\C, \nu)$, then 
$$\lim_{n\to \infty} I_{1,n}=\int_{\partial\C} \varphi\, \frac{\langle Du , DG\rangle _{|G^{-1}(0)}}{\|Q^{1/2}DG\|}\,e^{-2U_{|G^{-1}(0)} }d\rho .$$
Choosing $q\in (1, p(p_0-1)/p_0)$ and using the H\"older inequality with respect to the measure $e^{-2U_{\alpha_n}}  \rho$
we get
$$\begin{array}{l}
|I_{2,n}|     \leq     \|\varphi\|_{\infty}
\ds  \int_{G^{-1}(0)} \frac{|\langle Du_{\alpha_n} , DG\rangle |}{\|Q^{1/2}DG\|}\,(1- e^{-2U_{|G^{-1}(0)} + 2U_{\alpha_n}}) e^{-2U_{\alpha_n}}  d\rho 
\\ 
 \\
 \leq    \ds \|\varphi\|_{\infty}\bigg(  \int_{G^{-1}(0)}  |\langle Du_{\alpha_n} , DG\rangle |^q e^{-2U_{\alpha_n} } d\rho \bigg)^{1/q}\bigg(  \int_{G^{-1}(0)} \bigg(\frac{ 1- e^{-2U_{|G^{-1}(0)} + 2U_{\alpha_n}}}{\|Q^{1/2}DG\|} \bigg)^{q'} e^{-2U_{\alpha_n}}  d\rho \bigg)^{1/q'}.\end{array}
$$
Now we use Proposition \ref{Pr:traccia},  with $U$ replaced by $U_{\alpha_n}$. Estimate \eqref{maggtraccia} yields
$$\begin{array}{l} \ds \int_{G^{-1}(0)}   |\langle Du_{\alpha_n} , DG\rangle _{|G^{-1}(0)} | ^q e^{-2U_{\alpha_n} } d\rho\\
\\
  \leq 
C_0  \|\langle Du_{\alpha_n} , DG\rangle \|_{W^{1,p}_{1/2}(\C,\nu_{\alpha_n})}^q e^{-2C_{\alpha_n} (p-q)/p}(1+\|Q^{1/2} DU_{\alpha_n}\|\,\|_{L^{p_0}(X, \mu)}). \end{array}$$
By Proposition \ref{difficile}(i),   $\|\langle Du_{\alpha_n} , DG\rangle \|_{W^{1,p}(\C, \nu_{\alpha_n})}$ is  bounded by a constant independent of $n$. Moreover, 
 $U_{\alpha_n}(x) \geq C$ for every $x$, so that $e^{-2C_{\alpha_n} (p-q)/p} \leq 
e^{-2C  (p-q)/p} $, and $\|\, \|Q^{1/2}DU_{\alpha_n}\|\,\|_{L^{p_0} (X, \mu)}$  is bounded by a constant independent of $n$ by Hypothesis \ref{Hyp}. 

On the other hand the integral $ \int_{G^{-1}(0)} ((1- e^{-2U_{|G^{-1}(0)} + 2U_{\alpha_n}})/\|Q^{1/2}DG\|)^{q'} e^{-2U_{\alpha_n}}  d\rho $ vanishes by dominated convergence as $n\to \infty$, since $e^{-2U_{|G^{-1}(0)} + 2U_{\alpha_n}}\to 0$ $\rho$-a.e. in $G^{-1}(0)$,  $1- e^{-2U + 2U_{\alpha_n}}\in [0,1]$, $e^{-2U_{\alpha_n}} \leq e^{-2C } $, and $1/ \|Q^{1/2}DG\|\in L^s(G^{-1}(0), \rho)$ for every $s$. 
Therefore, 
$$\lim_{n\to \infty} I_{2,n}=0.$$
So, letting $n\to \infty$ in \eqref{eq:ualphan} we get
$$\begin{array}{lll}
\ds \int_{\C} (\lambda u -f)\,\varphi \,e^{-2U }d\mu& =&\ds - \frac{1}{2} \int_{\C} \langle Du , D\varphi\rangle e^{-2U }d\mu \\
\\
&&\ds +  \frac{1}{2}\int_{G^{-1}(0)}\varphi\, \frac{\langle Du , DG\rangle _{|G^{-1}(0)}}{\|Q^{1/2}DG\|}\,e^{-2U_{|G^{-1}(0)} }d\rho ,
 \end{array}$$
and since $u$ is a weak solution to \eqref{Neumann}, then $\int_{G^{-1}(0)} \varphi\, \frac{\langle D u, DG\rangle _{|G^{-1}(0)}}{\|Q^{1/2}DG\|}\,e^{-2U_{|G^{-1}(0)} }d\rho =0$.

 \vspace{3mm}
\noindent{\em Second step: if $f\in L^2(\C, \nu)$,  \eqref{intbordo} holds.}

Approaching the null extension of $f$ to the whole $X$ by a sequence of functions $f_n\in C^1_b(X)$, the sequence of the solutions $u_n$ to \eqref{Neumann} with datum $f_n$ converge to $u$ in $W^{2,2}(\C, \nu)$ $\cap$ $W^{1,2}_{-1/2}(\C, \nu)$, by Corollary \ref{Cor:u}. By Proposition \ref{difficile}(ii), the sequence $(\langle Du_n, DG\rangle)$ converge to $\langle Du, DG\rangle$
in $W^{1,p}_{1/2}(\C, \nu)$,  for every $p\in [p'_0, 2)$.

For every $n$ we have $\int_{G^{-1}(0)} \varphi\, \frac{\langle D u_n, DG\rangle _{|G^{-1}(0)}}{\|Q^{1/2}DG\|}\,e^{-2U_{|G^{-1}(0)} }d\rho =0$, and 
since the mapping \eqref{dual}
is continuous from $W^{1,p}(\C, \nu)$ to $\R$, letting $n\to \infty$ yields that $u$ satisfies \eqref{intbordo}. 

\vspace{3mm}
\noindent{\em Third step: $\langle Du , DG\rangle _{|G^{-1}(0)}=0$, $\rho$--a.e.}

Let $x\in X$, $r>0$, and let $(\varphi_n)$ be  a sequence of nonnegative functions belonging to $C^1_b(X)$, that converge monotonically to $\one_{B(x, r)}$. Then, 
$$\begin{array}{ll}
0 & =   \ds \lim_{n\to \infty} \int_{G^{-1}(0)} \varphi_n\, \frac{|\langle Du, DG\rangle _{|G^{-1}(0)}|}{\|Q^{1/2}DG\|}\,e^{-2U_{|G^{-1}(0)}}d\rho 
\\
\\
& = \ds \int_{G^{-1}(0) \cap B(x,r)}  \frac{|\langle Du , DG\rangle _{|G^{-1}(0)}|}{\|Q^{1/2}DG\|}\,e^{-2U_{|G^{-1}(0)}}d\rho ,
\end{array}$$
and since $d\rho$ is a Borel measure, $|\langle Du , DG\rangle _{|G^{-1}(0)}| \,e^{-2U_{|G^{-1}(0)}} =0$, $\rho$--a.e.
By Lemma \ref{Le:U}, $e^{-2U_{|G^{-1}(0)}} $ cannot vanish on a set with positive surface measure. It follows that $ \langle Du, DG\rangle _{|G^{-1}(0)} =0$, $\rho$--a.e.
 \end{proof}

\section{Applications}

\subsection{Admissible sets $\C$.}
Admissible sets $\C$ are for instance halfplanes such as $\C= \{x\in X:\; \langle x, y\rangle \leq  c\}$, for any $y\in X$ and $c\in \R$, balls and ellipsoids 
such as $\C= \{x\in X:\; \sum_{k\in \N} \alpha_k x_k^2 \leq  r^2\}$, where $(\alpha_k)$ is any bounded sequence with positive values. In these cases, $G(x) = r^2 - \sum_{k\in \N} \alpha_k x_k^2$ is smooth and Hypothesis \ref{HypG} is easily seen to hold. See \cite{Tracce}. 
 
We could also take an unbounded sequence $\alpha_k$, still satisfying 
\begin{equation}
\label{ulteriore}
\sum_{k=1}^{\infty}  \alpha_k^2\lambda_k<+\infty . 
\end{equation}
Indeed, in this case we have also $\sum_{k=1}^{\infty}  \alpha_k \lambda_k<+\infty  $, so that $G$ is $C_{2,p}$-quasicontinuous function for every $p>1$, and Hypothesis \ref{HypG}(i) is satisfied by \cite[sect. 5.3]{Tracce}. 
Moreover it is easy to see that $\C = G^{-1}(-\infty, 0]$ is convex and closed. However, in this case $G$ is not continuous, and the interior part of $\C$ is empty.

Another class of admissible domains, that may be seen as generalization of halfplanes,  are the regions below graphs of concave functions. For every $k\in \N$ set $X= $ span $e_k \oplus Y_k$, where $Y_k$ is the orthogonal complement of the linear span of $e_k$. The measure $\mu$ may be seen as the product measure of two Gaussian measures on span $e_k $ and on $Y_k$, precisely $\mu \circ \Pi_k^{-1}$ and $\mu_Y:= \mu \circ (I-\Pi_k)^{-1}$, where $\Pi_k$ is the orthogonal projection on the linear span of $e_k$, $\Pi_kx = \langle  x, e_k\rangle e_k = x_k e_k$. 

For every concave $F:Y_k\mapsto \R$, the set $\C = \{x: \;x_k \leq  F((I-\Pi_k)x)\}$ is convex. If in addition $F\in W^{2,p}_{1/2}(Y, \mu_Y) \cap W^{1,p}_{0}(Y, \mu_Y)$ for every $p>1$, then the function $G(x)= x_k- F((I-\Pi_k)x)$ satisfies Hypothesis \ref{HypG}. Indeed, it is continuous, it belongs to 
$W^{2,p}_{1/2}(X, \mu) \cap W^{1,p}_{0}(X, \mu)$ for every $p>1$, and $\|Q^{1/2}DG\| \geq \lambda_k^{1/2}$, so that $1/\|Q^{1/2}DG\|\in L^{p}(X, \mu)$ for every $p>1$.

\subsection{Kolmogorov equations of stochastic reaction--diffusion systems.}
\label{reactoin-diffusion}

We choose here  $X=L^2((0,1), d\xi)$, and $D(A)=W^{2,2}((0, \pi), d\xi) \cap W^{1,2}_0((0, \pi), d\xi)$, $Ax = x''$. 
$X$ is endowed with the Gaussian measure $\mu $  with mean $0$ and covariance $Q := -\frac12\,A^{-1}$. As   orthonormal basis of $X$ 
we choose $\{e_k(\xi) := \sqrt{2}\sin(k \pi \xi), \;k\in \N \}$ that
consists of  eigenfunctions of $Q$ with eigenvalues $\lambda_k = 1/(2k^2 \pi^2)$. 

The function $U$ is defined by 
\begin{equation}
\label{e4.2}
U(x)= \left\{ \begin{array}{lll}  & \ds \int_0^{1} \Phi(x(\xi))d\xi, & x\in L^{p}(0,1), 
\\
\\
 & +\infty , & x \notin L^{p}(0,1). 
\end{array}\right. 
\end{equation}
where $\Phi:\R\mapsto \R$ is a $C^1$ convex lowerly bounded function, such that 
\begin{equation}
\label{e4.4}
 |\Phi'(t)| \leq C(1+ |t|^{p-1}), \quad  t\in \R, 
 \end{equation}
for some $C>0$, $p\geq 1$. Note that \eqref{e4.4} implies $\Phi (t)\leq C_1(1+ |t|^{p})$ for every $t$, so that $U(x)<+\infty$ for every $x\in L^p(0,1)$. 

In the paper \cite{DPL0} we proved that $U$ satisfies the hypothesis of Lemma \ref{regU-Moreau}, with any $p_1>1$ and
$$D_0U(x) = \Phi'(x), \quad \mu-a.e. \;x\in X. $$
Therefore, Hypothesis \ref{Hyp} is satisfied taking  as $U_{\alpha}$ the Moreau--Yosida approximations of $U$. 
 
If $G:X\mapsto \R$ satisfies Hypothesis \ref{HypG}, the results of Corollary \ref{Cor:u} and of Theorem \ref{finale} hold. Namely, the weak solution $u$ to  \eqref{Neumann} in $\C = G^{-1}((-\infty, 0])$ belongs to $W^{2,2}(\C, \nu)\cap W^{1,2}_{-1/2}(\C, \nu)$,  it satisfies \eqref{stimaNeu}, and 
$\langle Du, DG\rangle _{|G^{-1}(0)}=0$, $\rho$-a.e. 

In this setting, \eqref{Neumann} is the Kolmogorov equation of the reaction--diffusion problem
$$\left\{ 
\begin{array}{l}
\displaystyle{dX = \bigg(\frac{\partial^2 X}{\partial \xi^2} -\Phi'(X)\bigg)}dt+ N_{\mathcal C}(X(t))dt \ni dW(t), \quad t>0,\; \xi\in (0,1),
\\
\\
X(t,0)= X(t,1)=0, \quad t>0,
\\
\\
X(0,x)=x,\quad   \xi\in (0,1).
 \end{array}\right.$$

\subsection{Kolmogorov equations of Cahn--Hilliard type equations.}
\label{Cahn--Hilliard}

Cahn--Hilliard type operators are characterized by a fourth order linear part and a nonlinearity of the type $u\mapsto \partial^2/\partial \xi^2(f\circ u )$.  
In the above section we have interpreted the nonlinearity $x\mapsto \Phi'\circ x$ as the gradient of a suitable function in the space $X= L^2(0,1)$. For a nonlinearity of the type $x\mapsto  \partial ^2/\partial \xi ^2 ( \Phi'\circ x)$ be a gradient, we have to change reference space and replace $L^2(0,1)$ by a  Sobolev space with negative exponent. It is convenient to work with functions with null average, setting
$$\overline{x} = \int_0^1 x(\xi)d\xi, \quad x\in L^2(0,1), $$
$$\mathcal{H} := \{ x\in H^1(0, 1): \;\overline{x} =0\}, \quad \|x\|_{\mathcal{H}} := \|x'\|_{L^2(0,1)}. $$
We take as $X$ the dual space of $\mathcal{H}$, endowed with the dual norm. We consider the spaces $L^p(0,1)$ as   subspaces of $X$, identifying any $x\in L^p(0,1)$ with the element $y\mapsto \int_0^1 x(\xi)y(\xi)d\xi$ of $X$. 

A realization of the negative second order derivative is a canonical isometry from $\mathcal{H} $ to $X$. More precisely, for every $x\in \mathcal{H}$ we define
$$Bx (y) =  \int_0^1 x'(\xi)y'(\xi) d\xi, \quad y\in \mathcal{H}, $$
so that for every $x\in \mathcal{H}$ we have $\|Bx\|_{X} = \sup_{y\neq 0} 
\langle  x, y\rangle _{\mathcal{H}} /\|y\|_{\mathcal{H}} = \|x\|_{\mathcal{H}}$.  If $x\in H^2(0,1)\cap \mathcal{H}$ and $x'(0)= x'(1)=0$, then $Bx(y) = -\langle x'', y\rangle _{L^2(0,1)}$ for every $y\in \mathcal{H}$. Therefore $B$ may be seen as an extension to $\mathcal{H}$ of the negative second order derivative with Neumann boundary condition.   It follows that if $y\in X$ and $g\in L^2(0,1)$, then $\langle B^{-1}y,g\rangle _{L^2(0,1)} = \langle  y,g\rangle _{X}$. 

The functions $e_k(\xi)=  \sqrt{2}\cos(k\pi\xi)/k\pi $, $k\in \N$, constitute an orthonormal basis of $\mathcal{H}$, and therefore, setting  $ f_k:= Be_k = k^2\pi^2 e_k$, the set $\{ f_k:\; k\in \N\} $  is an orthonormal basis of $X$. The operator $A := -B^2: D(B^2)\mapsto X$ is a realization of the negative fourth order derivative with null boundary condition for the first and third order derivatives in $X$, and we have $A^{-1}f_k = - f_k/k^4\pi^4$. Therefore $Q:= -A^{-1}/2$ is of trace class, and the Gaussian measure $\mu$ in $X$ with mean $0$ and covariance $Q$ is well defined. 

As in section \ref{reactoin-diffusion}, let $\Phi:\R\mapsto \R$ be any regular convex lowerly bounded function, satisfying \eqref{e4.4}, and let $U$ be defined by \eqref{e4.2}. It is possible to see  that $U\in  W^{1,q}_{1/2}(H, \mu)$ for every $q>1$, while in general
$U\notin  W^{1,2}_0(H, \mu)$. The proof given in \cite{DPL0} in a slightly different context works also in the present situation. Also, rephrasing the proof of Prop. 6.2 and Cor. 6.3 of \cite{DPL0} yields the following lemma, 

\begin{Lemma}
\label{Le:ellittiche}
For every $p\geq 1$, $\mu(\{ x\in L^p(0,1):\;\overline{x}=0\})=1$, and in addition $\int_{X} \|x\|_{L^p(0,1)}^q d\mu <+\infty $ for every $q>1$.  
\end{Lemma}

Let us check that Hypothesis \ref{Hyp} is satisfied. The approximating functions $U_{\alpha}$ are constructed as in \cite{DPDT}, namely we consider the Moreau--Yosida approximations of $\Phi$, 
$$\Phi_{\alpha} ( r) = \inf\{ \Phi(s) + (r-s)^2/2\alpha : s\in \R\}, \quad r\in \R, $$
and we define, for $\alpha >0$, 
$$U_{\alpha}(x) = \int_0^1 \Phi_{\alpha}((I+\alpha B)^{-1}x (\xi))d\xi, \quad  x\in X. $$

\begin{Lemma}
Let $\Phi \in C^2(\R)$ be a convex and lowerly bounded function satisfying 
\begin{equation}
\label{e4.5}
 |\Phi''(t)| \leq K(1+ |t|^{p-2}), \quad  t\in \R, 
 \end{equation}
for some $K>0$, $p\geq 2$. 
Then the functions  $U_{\alpha} $ satisfy  Hypothesis \ref{Hyp}. 
\end{Lemma}
\label{Le:Phi}
\begin{proof}
We see immediately that each $U_{\alpha}$ is convex,  lowerly bounded by $C:=\inf \Phi$, and of class $C^2$. Moreover for every $x$, $y\in X$ we have
$$\begin{array}{lll}
DU_{\alpha}(x) (y) & = & \displaystyle{\int_0^1 \Phi_{\alpha}'((I+\alpha B)^{-1}x (\xi)) (I+\alpha B)^{-1}y (\xi)\,d\xi}
\\
\\
& = & \displaystyle{\int_0^1( B (I+\alpha B)^{-1}\Phi_{\alpha}'((I+\alpha B)^{-1}x ))(\xi) B^{-1}y (\xi)\,d\xi }
\\
\\
& = & \langle  B(I+\alpha B)^{-1}  \Phi_{\alpha}'((I+\alpha B)^{-1}x), y\rangle_X
\end{array}$$
so that
\begin{equation}
\label{gradX}
DU_{\alpha}(x)= B(I+\alpha B)^{-1}(\Phi_{\alpha}'((I+\alpha B)^{-1}x)  . 
\end{equation}
Since $\Phi_{\alpha}':\R\mapsto \R$ is Lipschitz continuous, so is $DU_{\alpha}:X\mapsto X$, and Hypothesis \ref{Hyp}(i) is satisfied.

An argument taken from \cite{DPDT} shows that $U_{\alpha}(x)\leq U(x)$ for almost every $x\in X$. Indeed, for $x\in L^2(0,1)$, 
$$(I+\alpha B)^{-1}x (\xi) = \int_0^1 k(\xi, s)x(s)ds, $$
where $k(\xi, s)\geq 0$ for each $(\xi, s)\in (0,1)^2$, and $ \int_0^1 k(\xi, s)ds =1$ for every $\xi$. Then, 
$$\Phi_{\alpha}((I+\alpha B)^{-1}x (\xi))\leq \Phi ((I+\alpha B)^{-1}x (\xi))\leq (I+\alpha B)^{-1}(\Phi (x))(\xi), $$
where the last inequality follows from the Jensen inequality. Then, Hypothesis \ref{Hyp}(ii) is satisfied. 

Now, let us prove that $U_{\alpha}(x)$ converges to $U(x)$ as $\alpha \to 0$, for a.e. $x\in X$. Note that 
\eqref{e4.5} implies 
\begin{equation}
\label{comp}
(i)\; |\Phi'(t)| \leq C_1(1+ |t|^{p-1}),\quad (ii)\;|\Phi(t)| \leq C_0(1+ |t|^{p}), \quad  t\in \R, 
 \end{equation}
for some $C_1$, $C_0>0$. Moreover,  \eqref{comp}(ii) implies that $\Phi\circ x\in L^1(0,1)$, for every $x\in L^p(0, 1)$. 
Moreover, let us recall that  
$\lim_{\alpha\to 0}(I+\alpha B)^{-1}x =x$ in $L^p(0,1)$, and $\|(I+\alpha B)^{-1}x\|_{L^p(0,1)}\leq \|x\|_{L^p(0,1)}$ for every $x\in L^p(0,1)$ with zero average. 

Let  $x\in L^p(0,1)$ have zero average. Then, we split
 $\Phi_{\alpha} ((I+\alpha B)^{-1}x(\xi)) - \Phi(x(\xi)) = f_{\alpha}(\xi) + g_{\alpha}(\xi)$, with 
$f_{\alpha}(\xi) := \Phi_{\alpha} ((I+\alpha B)^{-1}x(\xi)) - \Phi_{\alpha}(x(\xi))$, $g_{\alpha}(\xi):= \Phi_{\alpha}(x(\xi))- \Phi(x(\xi))$, and using \eqref{comp}(i) we get 
$$\begin{array}{l}
|f_{\alpha}(\xi)| = \displaystyle{\bigg| \int_0^1} [\Phi_{\alpha}'(\sigma (I+\alpha B)^{-1}x(\xi)) + (1-\sigma)x(\xi)]d\sigma \, ((I+\alpha B)^{-1}x(\xi)- x(\xi))\bigg|
\\
\\
\leq C_1(1 + 2^{p-2}|(I+\alpha B)^{-1}x(\xi)|^{p-1} + |x(\xi)|^{p-1})\,|(I+\alpha B)^{-1}x(\xi)- x(\xi)|
\end{array}$$
so that $f_{\alpha}\in L^1(0,1)$. Using the H\"older inequality, we get 
$$\begin{array}{l}
\|f_{\alpha}\|_{L^1(0,1)} \leq C_1\|(I+\alpha B)^{-1}x - x\|_{L^1(0,1)} + 
\\
\\
+ C_1 2^{p-2}(\|(I+\alpha B)^{-1}x\|_{L^p(0,1)}^{p-1} + \|x\|_{L^p(0,1)}^{p-1})\|(I+\alpha B)^{-1}x -x\|_{L^p(0,1)}
\\
\\
\leq C_1\|(I+\alpha B)^{-1}x - x\|_{L^1(0,1)} + C_1 2^{p-1}\|x\|_{L^p(0,1)}^{p-1}\|(I+\alpha B)^{-1}x -x\|_{L^p(0,1)}.
\end{array}$$
Therefore, $\|f_{\alpha}\|_{L^1(0,1)} $ vanishes as $\alpha\to 0$. 
Moreover, $\|g_{\alpha}\|_{L^1(0,1)} $ vanishes too as $\alpha\to 0$, by monotone convergence. This implies that  $U_{\alpha}(x)$ converges to $U(x)$ as $\alpha \to 0$, for all $x\in L^p(0,1)$ with null average. 

Now we claim that $U_{\alpha} $ converges to $U $ as $\alpha \to 0$, in $L^q(X, \mu)$, for every $q>1$. 
We have
$$\int_X \bigg|\int_0^1 (\Phi_{\alpha} ((I+\alpha B)^{-1}x(\xi)) - \Phi(x(\xi)) )d\xi \bigg|^q\mu(dx) 
\leq 2^{q-1}\int_X (C_1\|(I+\alpha B)^{-1}x - x\|_{L^1(0,1)} $$
$$ +C_1 2^{p-1}\|x\|_{L^p(0,1)}^{p-1}\|(I+\alpha B)^{-1}x -x\|_{L^p(0,1)} )^q\mu(dx) 
+ 2^{q-1}\int_X \bigg(\int_0^1 (\Phi(x(\xi))- \Phi_{\alpha} (x(\xi)))d\xi \bigg)^q\mu(dx) .$$
The first integral  vanishes as $\alpha \to 0$, by Lemma \ref{Le:ellittiche} and dominated convergence. The second integral vanishes by monotone convergence, and the claim follows. 

Similar arguments yield that $Q^{1/2}D U_{\alpha}(x) = B^{-1}D U_{\alpha}(x) /\sqrt{2}$ converges to $ \Phi'(x)/\sqrt{2}$ pointwise a.e. and in $L^q(X, \mu)$, for every $q>1$. Indeed, by \eqref{gradX} we have $B^{-1}D U_{\alpha}(x) = (I+\alpha B)^{-1}( \Phi_{\alpha}'(I+\alpha B)^{-1}x)  $ for every $x\in X$. Arguing as before, with $ \Phi_{\alpha}'$ replacing $ \Phi_{\alpha}$, we see that $\Phi_{\alpha}'((I+\alpha B)^{-1}(x))$ converges to $\Phi '(x)$ in $L^1(0,1)$ as $\alpha \to 0$, for every $x\in L^{p-1}(0,1)$. (Note that now $g_{\alpha}$ does not converge to $0$ by monotone convergence but by dominated convergence, recalling that $\Phi_{\alpha}'\circ x$ converges to $\Phi'\circ x$  pointwise, and $|\Phi_{\alpha}' (x (\xi)) - \Phi'(x(\xi))| \leq | \Phi'(x(\xi))| \leq C_1(1+|x(\xi)|^{p-1})$). Recalling that the part of  $(I+\alpha B)^{-1}$ in $L^1(0,1)$ is a contraction in $L^1(0,1)$ we obtain that $ (I+\alpha B)^{-1}( \Phi_{\alpha}'(I+\alpha B)^{-1}x)  $ converges to $ \Phi'(x) $ in $L^1(0,1)$. 
Moreover, the last estimate yields $\|B^{-1}D U_{\alpha}\|_{L^1(0,1)} \leq   C_1(1+\|x \|_{L^{p-1}(0,1)}^{p-1})$. Since $L^1(0,1)$ is continuously embedded in $X$, Lemma \ref{Le:ellittiche} implies that $\|B^{-1}D U_{\alpha}(x)\| \leq g(x)$ for some $g\in \cap_{q>1}L^q(X, \mu)$ and for every $x\in L^{p-1}(0,1)$, hence for $\mu$-a.e. $x\in X$. 
By dominated convergence, $Q^{1/2}D U_{\alpha} $ converges to $ \Phi'(\cdot ) /\sqrt{2}$ in $L^q(X, \mu; X)$, for every $q\geq 1$. This shows that $U\in W^{1,q}_{1/2}(X, \mu)$, and ends the proof. 
\end{proof}

So, we can consider Kolmogorov operators of the Cahn--Hilliard  equation with reflection
$$\left\{ 
\begin{array}{l}
\displaystyle{dX = \frac{d^2}{d\xi^2}\bigg( - \frac{d^2X}{d\xi^2} + \Phi'(X)\bigg)dt +  N_{\mathcal C}(X(t))dt \ni dW(t), }\quad t>0, \;\xi\in (0,1), 
\\
\\
\ds\int_0^1X(\xi)d\xi =0, \quad t>0, 
\\
\\
\ds \frac{d X}{d\xi }(t, 0) = \frac{d^3 X}{d\xi ^3}(t, 0) = \frac{d X}{d\xi }(t, 1) = \frac{d^3 X}{d\xi ^3}(t, 1) =0, \quad t>0, 
\\
\\
X(0, \cdot) = x, 
\end{array}\right. $$
provided $\Phi$ satisfies the assumptions of Lemma \ref{Le:Phi}.  If in addition $G:X\mapsto \R$ satisfies Hypothesis \ref{HypG},  Corollary \ref{Cor:u} and   Theorem \ref{finale} yield that the weak solution $u$ to  \eqref{Neumann} in $\C = G^{-1}((-\infty, 0])$ belongs to $W^{2,2}(\C, \nu)\cap W^{1,2}_{-1/2}(\C, \nu)$,  it satisfies \eqref{stimaNeu}, and 
$\langle Du, DG\rangle _{|G^{-1}(0)}=0$, $\rho$-a.e.

\end{document}